\documentclass[a4paper,11pt,intlimits,oneside]{amsart}

\usepackage{amssymb,latexsym,amsmath,mathrsfs}
\usepackage{amsfonts,amsbsy,bm}
\usepackage{color}
\usepackage[margin=2cm]{geometry}

\usepackage{amsmath}
\usepackage{amssymb}
\usepackage{amsfonts}
\usepackage{amsthm,amscd}

\newtheorem{theorem}{Theorem}[section]
\newtheorem{proposition}[theorem]{Proposition}
\newtheorem{corollaire}[theorem]{Corollary}
\newtheorem{lemme}[theorem]{Lemma}
\newtheorem{definition}[theorem]{Definition}

\theoremstyle{definition}
\newcommand{\comment}[1]{}

\numberwithin{equation}{section}

\newcommand{\epf}{ $\Box$\medskip}

\theoremstyle{definition}

\begin{document}
\title[J.M Tanoh Dje and B. Sehba]{Carleson embeddings and pointwise multipliers between Hardy-Orlicz spaces and Bergman-Orlicz spaces of the upper half-plane}
\author[J.M. Tanoh Dje]{Jean$-$Marcel Tanoh Dje}
\address{Unit\'e de Recherche et d'Expertise Num\'erique, Universit\'e Virtuelle de C\^ote d'Ivoire, Cocody II-Plateaux - 28 BP 536 ABIDJAN 28}
\email{{\tt tanoh.dje@uvci.edu.ci}}
\author{Beno\^it Florent Sehba}
\address{Department of Mathematics, University of Ghana,\\ P. O. Box LG 62 Legon, Accra, Ghana}
\email{bfsehba@ug.edu.gh}

\subjclass{}
\keywords{}

\date{}

\begin{abstract}
In this article, we give a general characterization of Carleson measures involving concave or convex growth functions. We use this characterization to establish continuous injections and also to characterize the set of pointwise multipliers between Hardy-Orlicz spaces and Bergman-Orlicz spaces.
\end{abstract}

\maketitle

\section{Introduction.}
Let $\mathbb{D}$ be the unit disc of $\mathbb{C}$. For $\alpha>-1$, and $0<p<\infty$, the Bergman space $A_\alpha^p(\mathbb{D})$ consists of all holomorphic functions $f$ on $\mathbb{D}$ such that \begin{equation}\label{eq:defbergdisc}
\|f\|_{p,\alpha}^p:=\int_{\mathbb D}\vert f(z)\vert^p(1-|z|^2)^\alpha d\nu(z)<\infty.
\end{equation}
Here, $d\nu(z)$ is the normalized area measure on $\mathbb D$.
\medskip

When $\alpha\longrightarrow -1$, the corresponding space $A_{-1}^p(\mathbb{D})$ is the Hardy space $H^p(\mathbb D)$ which consists of all holomorphic functions $f$ on $\mathbb D$ such that 
\begin{equation}\label{eq:defhardydisc}
\|f\|_p^p:=\|f\|_{p,-1}^p:=\sup_{0\le r<1}\int_0^{2\pi}\vert f(re^{i\theta})\vert^pd\theta<\infty.
\end{equation}
\medskip

One of the most studied questions on holomorphic function spaces and their operators is the notion of Carleson meausures for these spaces. In the unit disc, this is about characterizing all positive measures $\mu$ on $\mathbb D$ such that for some constant $C>0$, and for any $f\in A_\alpha^p(\mathbb{D})$, $\alpha\geq -1$,
\begin{equation}\label{eq:carlembeddisc}
\int_{\mathbb D}\vert f(z)\vert^qd\mu(z)\leq C\|f\|_{p,\alpha}^q.
\end{equation}
This problem was first solved by L. Carleson in \cite{Carleson1,Carleson2} for Hardy spaces in the case $p=q$. Extension of this result for $p<q$ was obtained by P. Duren in \cite{Pduren2} The case with loss $p<>q$ was solved by I. V. Videnskii in \cite{Videnskii}. The Corresponding results for Bergman spaces of the unit disc and the unit ball were obtained by W. Hastings and D. Luecking, J. A. Cima and W. Wogen in \cite{CW,Hastings,Luecking1,Luecking2,Luecking3, Luecking4}. For other contributions, we also refer the reader to the following \cite{hormander,Power,Ueki}.  
\medskip

Our interest in this paper is for the inequality (\ref{eq:carlembeddisc}) in the case where the power functions $t^q$ and $t^p$ are replaced by some continuous increasing and onto functions on $[0,\infty)$, $\Phi_2$ and $\Phi_1$ respectively. In the unit ball of $\mathbb C^n$, this problem was solved in the case where $t\mapsto \frac{\Phi_2(t)}{\Phi_1(t)}$ is nondecreasing for Hardy and Bergman spaces in the following and the references therein \cite{Charpentier,CharpentierSehba,Sehba2}. The  case where $t\mapsto \frac{\Phi_2(t)}{\Phi_1(t)}$ is nonincreasing was handled in \cite{Sehba1} for the Bergman-Orlicz spaces.
\medskip 

In this paper, our setting is the upper-half plane $\mathbb C_+$ and we still consider problem (\ref{eq:carlembeddisc}) for growth functions $\Phi_1$ and $\Phi_2$. In \cite{djesehb}, we considered this question for the case where $t\mapsto \frac{\Phi_2(t)}{\Phi_1(t)}$ is nondecreasing both functions being convex growth function. We are presenting here a more general result that encompasses the case where both $\Phi_1$ and $\Phi_2$ are concave, still with $t\mapsto \frac{\Phi_2(t)}{\Phi_1(t)}$ nondecreasing. We note that even in the case of power functions, the study of Carleson measures for Bergman spaces of the upper-half plane with exponent in $(0,1]$ seems to have never been considered before. Our work will fix this gap beyond power functions as we are dealing here with growth functions that generalize them.
\section{Statement of main results.}

In this paper, a continuous and nondecreasing function $\Phi$ from $\mathbb{R}_{+}$ onto itself is called a growth function. Observe that if $\Phi$ is a growth function, then $\Phi(0)=0$ and $\lim_{t \to +\infty}\Phi(t)=+\infty$. If  $\Phi(t)> 0$ for all  $t> 0$ then $\Phi$ is a homeomorphism of $\mathbb{R}_{+}$ onto $\mathbb{R}_{+}$.

Let  $p>0$ be a real and $\Phi$  a growth function. We say that  $\Phi$ is of upper-type (resp. lower-type) $p>0$ if there exists a constant
$C_{p}> 0$ such that for all $t\geq 1$ (resp. $0< t\leq 1$), 
\begin{equation}\label{eq:sui8n}
\Phi(st)\leq C_{p}t^{p}\Phi(s),~~\forall~s>0.\end{equation}
We denote by $\mathscr{U}^{p}$ (resp. $\mathscr{L}_{p}$) the set of all growth functions of upper-type $p \geq 1$ (resp. lower-type $0< p\leq 1$) such that the function 
$t\mapsto \frac{\Phi(t)}{t}$ is non decreasing (resp. non-increasing) on $\mathbb{R}_{+}^{*}= \mathbb{R}_{+}\backslash\{0\}$. We put
$   \mathscr{U}:=\bigcup_{p\geq 1}\mathscr{U}^{p}$ (resp. $\mathscr{L}:=\bigcup_{0< p\leq 1}\mathscr{L}_{p}$).\\
Any element belongs $\mathscr{L} \cup\mathscr{U}$ is a homeomorphism of $\mathbb{R}_{+}$ onto $\mathbb{R}_{+}$. 

We say that two growth functions $\Phi_{1}$ and $\Phi_{2}$ are equivalent, if there exists a constant $c > 0$ such that

\begin{equation}\label{eq:equivalent}
c^{-1}\Phi_{1}(c^{-1}t) \leq \Phi_{2}(t)\leq c\Phi_{1}(ct), ~~ \forall~ t > 0.\end{equation} 

We will assume in the sequel that  any element of $\mathscr{U}$ (resp. $\mathscr{L}$) belongs to $\mathscr{C}^{1}(\mathbb{R}_{+})$ and is  
convex (resp. concave). Moreover, 
$$  \Phi'(t)\approx \frac{\Phi(t)}{t},~~ \forall~ t > 0,    $$ 
(see for example \cite{bonamisehba, jmtafeuseh, djesehb, djesehaqb, sehbaedgc}).  

 Let $I$ be an interval of nonzero length. The Carleson square associated with $I$, $Q_{I}$ is the subset of $\mathbb{C}_{+}$ defined by

\begin{equation}\label{eq:ualp5rsleson}
Q_{I}:=\left\{x+iy\in \mathbb{C}_{+} : x\in I ~~\text{et}~~ 0<y<|I| \right\}.\end{equation}

\begin{definition}
Let  $s > 0$ be a real, $\Phi$ a growth function and  $\mu$ a positive Borel measure on $\mathbb{C}_{+}$. We say that $\mu$ is a $(s,\Phi)-$Carleson measure if there is a constant $C > 0$ such that for any interval $I$ of nonzero length
\begin{equation}\label{eq:ual65iaq5rsleson}
\mu(Q_{I}) \leq \dfrac{C}{\Phi\left(\frac{1}{|I|^{s}}\right)} .\end{equation}
\end{definition}

\begin{itemize}
\item[\textbullet] When $s=1$, we say that $\mu$ is a $\Phi-$Carleson measure.
\item[\textbullet] When $s=2+\alpha$, with $\alpha > -1$, we say that $\mu$ is a $(\alpha,\Phi)-$Carleson measure.
\end{itemize}

Let  $\alpha > -1$ be a real and $\Phi$ a  growth function. 
\begin{itemize}
\item The Hardy$-$Orlicz space on  $\mathbb{C_{+}}$,   $H^{\Phi}(\mathbb{C_{+}})$ is the set of analytic functions on $\mathbb{C_{+}}$ which satisfy 
$$   \|F\|_{H^{\Phi}}^{lux}:=\sup_{y> 0}\inf\left\{\lambda>0 : \int_{\mathbb{R}}\Phi\left(\dfrac{| F(x+iy)|}{\lambda}\right)dx \leq 1  \right\} < \infty.   $$  
\item The Bergman$-$Orlicz space on  $\mathbb{C_{+}}$, $A_{\alpha}^{\Phi}(\mathbb{C_{+}})$ is the set of analytic functions on $\mathbb{C_{+}}$ which satisfy 
$$ \|F\|_{A_{\alpha}^{\Phi}}^{lux}:=\inf\left\{\lambda>0 : \int_{\mathbb{C_{+}}}\Phi\left(\dfrac{| F(x+iy)|}{\lambda}\right)dV_{\alpha}(x+iy) \leq 1  \right\}< \infty, $$
where $dV_{\alpha}(x+iy):=y^{\alpha}dxdy$.
\end{itemize}
If $\Phi$ is  convex and  $\Phi(t)> 0$ for all  $t> 0$ then  $\left(H^{\Phi}(\mathbb{C_{+}}), \|.\|_{H^{\Phi}}^{lux}\right)$ and  $(A_{\alpha}^{\Phi}(\mathbb{C_{+}}),  \|.\|_{A_{\alpha}^{\Phi}}^{lux})$  are Banach spaces (see. \cite{djesehb, Jan, Jan1}).  The spaces  $H^{\Phi}(\mathbb{C_{+}})$  and $A_{\alpha}^{\Phi}(\mathbb{C_{+}})$  generalizes respectively the Hardy space $H^{p}(\mathbb{C_{+}})$ and the Bergman space $A_{\alpha}^{p}(\mathbb{C_{+}})$ for $0< p < \infty$.
\medskip

Our first main result is the following which extend  \cite[Theorem 2.2]{djesehb} to Hardy-Orlicz spaces defined with concave growth functions.
\begin{theorem}\label{pro:main127}
Let $\Phi_{1}, \Phi_{2} \in \mathscr{L} \cup\mathscr{U}$  and  $\mu$ a positive Borel measure on $\mathbb{C}_{+}$. 
 If the function   $t\mapsto\frac{\Phi_{2}(t)}{\Phi_{1}(t)}$ is non-decreasing on  $\mathbb{R_{+}^{*}}$ then the following assertions are equivalent.
\begin{itemize}
\item[(i)]  $\mu$ is a $\Phi_{2}\circ\Phi^{-1}_{1}-$Carleson measure.
\item[(ii)] There exist some constants $\rho \in \{1; a_{\Phi_{1}}\} $ and $C_{1}>0$  such that for all $z=x+iy \in \mathbb{C_{+}}$
\begin{equation}\label{eq:ualphispleson1}
 \int_{\mathbb{C}_{+}}\Phi_{2}\left(\Phi^{-1}_{1}\left(\frac{1}{y}\right) \dfrac{y^{2/\rho}}{ |\omega-\overline{z}|^{2/\rho}}\right) d\mu(\omega) \leq C_{1}  .\end{equation}
\item[(iii)] There exists a constant $C_{2} > 0$ such that for all $0\not \equiv F\in H^{\Phi_{1}}(\mathbb{C_{+}})$, 
\begin{equation}\label{eq:iberginjection}
\int_{\mathbb{C_{+}}}\Phi_{2}\left( \frac{|F(z)|}{\|F\|_{H^{\Phi_{1}}}^{lux}} \right)d\mu(z) \leq C_{2} .\end{equation}	
\item[(iv)]  There exists a constant $C_{3} > 0$ such that for all $ F\in H^{\Phi_{1}}(\mathbb{C_{+}})$
 \begin{equation}\label{eq:iberginjectio2}
\sup_{\lambda> 0}\Phi_{2}(\lambda)\mu\left(\{ z\in \mathbb{C_{+}} : |F(z)|>\lambda \|F\|_{H^{\Phi_{1}}}^{lux}\} \right) \leq C_{3} .\end{equation}
\end{itemize}
\end{theorem}
As consequence, we have the following.
\begin{corollaire}\label{pro:main149}
Let $\alpha > -1$ and $\Phi_{1},\Phi_{2} \in \mathscr{L} \cup \mathscr{U}$  such that $t\mapsto\frac{\Phi_{2}(t)}{\Phi_{1}(t)}$ is non-decreasing on  $\mathbb{R_{+}^{*}}$. The Hardy-Orlicz space $H^{\Phi_{1}}(\mathbb{C_{+}})$ embeds continuously into the Bergman-Orlicz space $A_{\alpha}^{\Phi_{2}}(\mathbb{C_{+}})$ if and only if there exists a constant $C > 0$ such that 
 for all $t > 0$, 
\begin{equation}\label{eq:ibejection122}
  \Phi_{1}^{-1}(t) \leq \Phi_{2}^{-1}(Ct^{2+\alpha})
.\end{equation}
\end{corollaire}
\medskip

Our second main result generalizes \cite[Theorem 2.4]{djesehb}.
\begin{theorem}\label{pro:main150}
Let  $\alpha > -1$,  $\Phi_{1}, \Phi_{2} \in \mathscr{L} \cup\mathscr{U}$  and  $\mu$ a positive Borel measure on $\mathbb{C}_{+}$. If the function   $t\mapsto\frac{\Phi_{2}(t)}{\Phi_{1}(t)}$ is non-decreasing on  $\mathbb{R_{+}^{*}}$ then the following assertions are equivalent.
\begin{itemize}
\item[(i)]  $\mu$ is a $(\alpha,\Phi_{2}\circ\Phi^{-1}_{1})-$Carleson measure.
\item[(ii)] There exist some constants $\rho \in \{1; a_{\Phi_{1}}\} $ and $C_{1}>0$  such that for all $z=x+iy \in \mathbb{C_{+}}$
\begin{equation}\label{eq:ualphispleson2}
\int_{\mathbb{C}_{+}}\Phi_{2}\left(\Phi^{-1}_{1}\left(\frac{1}{y^{2+\alpha}}\right) \dfrac{y^{(4+2\alpha)/\rho}}{ |\omega-\overline{z}|^{(4+2\alpha)/\rho}}\right) d\mu(\omega) \leq C_{1}  .\end{equation}
\item[(iii)] There exists a constant $C_{2} > 0$ such that for all $0\not\equiv F\in A_{\alpha}^{\Phi_{1}}(\mathbb{C_{+}})$, 
\begin{equation}\label{eq:iberginjectiber}
\int_{\mathbb{C_{+}}}\Phi_{2}\left( \frac{|F(z)|}{\|F\|_{A_{\alpha}^{\Phi_{1}}}^{lux}} \right)d\mu(z) \leq C_{2} .\end{equation}	
\item[(iv)]  There exists a constant $C_{3} > 0$ such that for all $F\in A_{\alpha}^{\Phi_{1}}(\mathbb{C_{+}})$  
 \begin{equation}\label{eq:ibeectio2}
\sup_{\lambda> 0}\Phi_{2}(\lambda)\mu\left(\{ z\in \mathbb{C_{+}} : |F(z)|>\lambda \|F\|_{A_{\alpha}^{\Phi_{1}}}^{lux}\} \right) \leq C_{3} .\end{equation}
\end{itemize}
\end{theorem}
The following embedding result follows from the above.
\begin{corollaire}\label{pro:main151}
Let $\alpha,\beta > -1$ and $\Phi_{1},\Phi_{2} \in \mathscr{L} \cup \mathscr{U}$  such that $t\mapsto\frac{\Phi_{2}(t)}{\Phi_{1}(t)}$ is non-decreasing on  $\mathbb{R_{+}^{*}}$. The Bergman-Orlicz space $A_{\alpha}^{\Phi_{1}}(\mathbb{C_{+}})$ embeds continuously into the Bergman-Orlicz space $A_{\beta}^{\Phi_{2}}(\mathbb{C_{+}})$ if and only if there exists a constant $C > 0$ such that 
 for all $t > 0$, 
\begin{equation}\label{eq:ibejecti2}
  \Phi_{1}^{-1}(t^{2+\alpha}) \leq \Phi_{2}^{-1}(Ct^{2+\beta})
.\end{equation}
\end{corollaire}

Let  $\Phi \in \mathscr{C}^{1}(\mathbb{R}_{+})$  a growth function. The lower and the upper indices of $\Phi$ are respectively defined by
$$ a_\Phi:=\inf_{t>0}\frac{t\Phi'(t)}{\Phi(t)}
  \hspace*{1cm}\textrm{and} \hspace*{1cm} b_\Phi:=\sup_{t>0}\frac{t\Phi'(t)}{\Phi(t)}.          $$

Let $p, q >0$ and  $\Phi$  a growth function. We say that $\Phi$ belongs to $\widetilde{\mathscr{U}}^{q}$ (resp. $\widetilde{\mathscr{L}}_{p}$) if the following assertions are satisfied 
\begin{itemize}
\item[(a)]   $\Phi \in \mathscr{U}^{q}$ (resp. $\Phi \in \mathscr{L}_{p}$).
\item[(b)] there exists a constant $C_{1} > 0$ such that 
 for all $s,t> 0$, 
\begin{equation}\label{eq:ibejecaqti2}
  \Phi(st) \leq C_{1}\Phi(s)\Phi(t).\end{equation}
\item[(c)] there exists a constant  $C_{2}>0$  such that for all $s,t \geq 1$ 
\begin{equation}\label{eq:ibide2}
\Phi\left(\frac{s}{t}\right)\leq C_{2}\frac{\Phi(s)}{t^{q}} 
\end{equation}
resp.  
\begin{equation}\label{eq:iaqide2}
\Phi\left(\frac{s}{t}\right)\leq C_{2}\frac{s^{p}}{\Phi(t)}.
\end{equation}
\end{itemize}
We put $\widetilde{\mathscr{U}}:=\bigcup_{q\geq 1}\widetilde{\mathscr{U}}^{q}$ (resp. $\widetilde{\mathscr{L}}:=\bigcup_{0<p \leq 1}\widetilde{\mathscr{L}}_{p}$).

Let $\omega : \mathbb{R_{+}^{*}} \longrightarrow \mathbb{R_{+}^{*}}$ be a  function. An analytic function $F$ in $\mathbb{C_{+}}$ is said to be in  $H_{\omega}^{\infty}(\mathbb{C_{+}})$ if 
\begin{equation}\label{eq:ibet78ide12}
\|F\|_{H_{\omega}^{\infty}}:=\sup_{z\in \mathbb{C_{+}}}\frac{|f(z)|}{\omega(\mathrm{Im}(z))}< \infty
.\end{equation}
If $\omega$ is continuous then $(H_{\omega}^{\infty}(\mathbb{C_{+}}), \|.\|_{H_{\omega}^{\infty}})$ is a Banach space.

Let $X$ and $Y$ be two analytic function spaces which are metric spaces, with respective metrics $d_{X}$ and $d_{Y}$. An analytic function $g$ is said to be a multiplier from $X$ to $Y$, if there exists a constant $C> 0$ such that for any  $f \in X$,
\begin{equation}\label{eq:ibetide12}
d_{Y}(fg, 0) \leq Cd_{X}(f,0)
.\end{equation}
We denote by $\mathcal{M}(X, Y)$ the set of multipliers from $X$ to $Y$.
\medskip

The following is a characterization of pointwise multipliers from an Hardy-Orlicz space to a Bergman-Orlicz space. It is an extension of \cite[Theorem 2.7]{djesehb}.
\begin{theorem}\label{pro:main16qs13}
Let $\Phi_{1}\in \mathscr{L} \cup \mathscr{U}$ and $\Phi_{2}\in \widetilde{\mathscr{L}} \cup \widetilde{\mathscr{U}}$ such that the function   $t\mapsto\frac{\Phi_{2}(t)}{\Phi_{1}(t)}$ is non-decreasing on  $\mathbb{R_{+}^{*}}$. Let  $\alpha > -1$ and  put 
$$ \omega(t)=\dfrac{\Phi_{2}^{-1}\left(\frac{1}{t^{2+\alpha}}\right)}{\Phi_{1}^{-1}\left(\frac{1}{t}\right)},  ~~ \forall~ t  >0.   $$
The following assertions are satisfied.
\begin{itemize}
 \item[(i)] If $0< a_{\Phi_{1}}\leq b_{\Phi_{1}} <  a_{\Phi_{2}}\leq b_{\Phi_{2}}  <\infty$ then 
$$  \mathcal{M}( H^{\Phi_{1}}(\mathbb{C_{+}}) ,  A^{\Phi_{2}}_{\alpha}(\mathbb{C_{+}})) = H_{\omega}^{\infty}(\mathbb{C_{+}}).    $$ 
\item[(ii)] If $\omega \approx 1$  then  
$$  \mathcal{M}( H^{\Phi_{1}}(\mathbb{C_{+}}) ,  A^{\Phi_{2}}_{\alpha}(\mathbb{C_{+}})) = H^{\infty}(\mathbb{C_{+}}).    $$
\item[(ii)] If $\omega$ is decreasing and $\lim_{t \to 0}\omega(t)=0$ then 
 $$  \mathcal{M}( H^{\Phi_{1}}(\mathbb{C_{+}}) ,  A^{\Phi_{2}}_{\alpha}(\mathbb{C_{+}})) = \{ 0 \}.    $$
\end{itemize}
\end{theorem}	

The following is a characterization of pointwise multipliers Bergman-Orlicz spaces. It is an extension of \cite[Theorem 2.8]{djesehb}.
\begin{theorem}\label{pro:main1sr606}
Let $\Phi_{1}\in \mathscr{L} \cup \mathscr{U}$ and $\Phi_{2}\in \widetilde{\mathscr{L}} \cup \widetilde{\mathscr{U}}$ such that the function   $t\mapsto\frac{\Phi_{2}(t)}{\Phi_{1}(t)}$ is non-decreasing on  $\mathbb{R_{+}^{*}}$. Let  $\alpha, \beta > -1$ and  put 
$$   \omega(t)=\dfrac{\Phi_{2}^{-1}\left(\frac{1}{t^{2+\beta}}\right)}{\Phi_{1}^{-1}\left(\frac{1}{t^{2+\alpha}}\right)},  ~~ \forall~ t  >0.  $$
The following assertions are satisfied.
\begin{itemize}
\item[(i)] If $0< a_{\Phi_{1}}\leq b_{\Phi_{1}} <  a_{\Phi_{2}}\leq b_{\Phi_{2}}  <\infty$ then
$$  \mathcal{M}\left( A_{\alpha}^{\Phi_{1}}(\mathbb{C_{+}}) ,  A^{\Phi_{2}}_{\beta}(\mathbb{C_{+}})\right) = H_{\omega}^{\infty}(\mathbb{C_{+}}).  $$
\item[(ii)]  If $\omega \approx 1$  then 
$$ \mathcal{M}\left( A_{\alpha}^{\Phi_{1}}(\mathbb{C_{+}}) ,  A^{\Phi_{2}}_{\beta}(\mathbb{C_{+}})\right) = H^{\infty}(\mathbb{C_{+}}).   $$
\item[(iii)] If $\omega$ is decreasing and $\lim_{t \to 0}\omega(t)=0$ then
$$  \mathcal{M}\left( A_{\alpha}^{\Phi_{1}}(\mathbb{C_{+}}) ,  A^{\Phi_{2}}_{\beta}(\mathbb{C_{+}})\right) = \{ 0 \}.  $$
\end{itemize}
\end{theorem}	
The paper is organized as folllows. In Section 3, we provide some further definitions and useful results on growth functions, Hardy-Orlicz and Bergman-Orliz spaces. Indeed, there is no actual  reference for a full study of our spaces in the literature, consequently, we are proving several related results needed in our study. In Section 4, we prove some characterizations of Carleson measures, in particular, a general result that encompasses assertions (ii) in both Theorem \ref{pro:main127} and Theorem \ref{pro:main150}. Our main results are proved in Section 5.

\section{Some definitions and useful properties}

We present in this section some useful results needed in our presentation.

\subsection{Some properties of growth functions.} 

Let  $\Phi$ be a growth function. We say that  $\Phi$ satisfies the $\Delta_{2}-$condition (or $\Phi \in \Delta_{2}$) if there exists a constant
$K > 1$ such that
\begin{equation}\label{eq:delta2}
\Phi(2t) \leq K \Phi(t),~ \forall~ t > 0.\end{equation}
It is obvious that any growth function  $\Phi \in \mathscr{L} \cup \mathscr{U}$ satisfies the  $\Delta_{2}-$condition. 

Let $\Phi$ be a convex growth function.  The complementary function of $\Phi$ is the function $\Psi$ defined by
$$ \Psi(s)=\sup_{t\geq 0}\{st-\Phi(t) \}, ~ \forall~  s \geq 0.       $$

Let $\Phi$ be a convex growth function. We say that $\Phi$ satisfies $\nabla_{2}-$condition (or $\Phi \in \nabla_{2}$)  if $\Phi$ and its complementary function both satisfy $\Delta_{2}-$condition.

Let $\Phi \in \mathscr{C}^{1}(\mathbb{R}_{+})$ a growth function. The following assertions are satisfied.
\begin{itemize}
\item[(i)] If  $\Phi \in \mathscr{L} \cup \mathscr{U}$  then  $0< a_\Phi\leq b_\Phi <\infty$.
\item[(ii)]  $\Phi \in \mathscr{U}$      if and only if  $1\leq a_\Phi\leq b_\Phi <\infty$. Moreover,  $\Phi \in \mathscr{U}\cap \nabla_{2}$  if and only if  $1< a_\Phi\leq b_\Phi <\infty$, (see. \cite{djesehb}).
\item[(iii)] If  $0< a_\Phi\leq b_\Phi <\infty$ then the function  $t\mapsto \frac{\Phi(t)}{t^{a_\Phi}}$ is increasing on $\mathbb{R}_{+}^{*}$ while the function  $t\mapsto \frac{\Phi(t)}{t^{b_\Phi}}$ is decreasing on $\mathbb{R}_{+}^{*}$ (see. \cite[Lemma 2.1]{sehbaedgc}).
\end{itemize}

Let $\Phi$ be a  growth function and $q > 0$. If $\Phi$ is a one-to-one growth function then $\Phi \in  \mathscr{U}^{q}$ if and only if $\Phi^{-1} \in  \mathscr{L}_{1/q}$ (see. \cite[Proposition 2.1]{sehbatchoundja1}).

\begin{lemme}[Lemma 3.1, \cite{djesehb}]\label{pro:main40l6aqlj}
Let  $\Phi \in \mathscr{U}$. The following assertions are equivalents.
\begin{itemize}
\item[(i)] $\Phi \in \nabla_{2}$.
\item[(ii)] There exists a constant $C_{1} >0$ such that for all $t>0$,  
\begin{equation}\label{eq:conditiondedinis}
\int_0^t\frac{\Phi(s)}{s^2}ds\le C_{1}\frac{\Phi(t)}{t}.\end{equation}
\item[(iii)] There exists a constant $C_{2}>1$ such that for all $t>0$,
\begin{equation}\label{eq:conditiondedin}
\Phi(t)\leq \frac{1}{2C_{2}}\Phi(C_{2}t). 
\end{equation}
\end{itemize}
\end{lemme}

\begin{lemme}\label{pro:main3aq7}
Let $\Phi \in \mathscr{C}^{1}(\mathbb{R}_{+})$ be a growth function such that  $0< a_\Phi \leq b_\Phi < \infty $. For $s >0$, consider $\Phi_{s}$ the function defined by
$$  \Phi_{s}(t)=\Phi\left(t^{s}\right), ~~\forall~t\geq 0.    $$
Then  $sa_\Phi \leq a_{\Phi_{s}} \leq b_{\Phi_{s}} \leq sb_\Phi.$ 
\end{lemme}

\begin{proof}
For   $t>0$, we have
 $$  \left( \Phi_{s}(t)  \right)'= st^{s-1}\Phi'\left(t^{s}\right) \Rightarrow   \frac{t\left( \Phi_{s}(t)  \right)'}{\Phi_{s}(t)} = s\times \frac{t^{s}\Phi'\left(t^{s}\right)}{\Phi\left(t^{s}\right)}.   $$
It follows that 
$$  s a_\Phi \leq \frac{t\left( \Phi_{s}(t)  \right)'}{\Phi_{s}(t)} \leq  s b_\Phi, ~~\forall~t >0.       $$
\end{proof}

\begin{corollaire}\label{pro:main3aaqq7}
Let $s \geq 1$  and   $\Phi \in \mathscr{C}^{1}(\mathbb{R}_{+})$ a growth function such that  $0< a_\Phi \leq b_\Phi < \infty $. For  $t\geq 0$, put 
$$  \Phi_{s}(t)=\Phi\left(t^{s/a_\Phi}\right).    $$
The following assertions are satisfied.
\begin{itemize}
\item[(i)] If $s=1$ then $\Phi_{s} \in \mathscr{U}$.
\item[(ii)] If $s>1$ then  $\Phi_{s} \in \mathscr{U} \cap \nabla_{2}$.
\end{itemize}
\end{corollaire}

\begin{proposition}\label{pro:main4paqm4}
Let  $\Phi_{1},\Phi_{2} \in \mathscr{C}^{1}(\mathbb{R}_{+})$ be two growth functions such that  $0< a_{\Phi_{1}}\leq b_{\Phi_{1}} <\infty$ and  $0< a_{\Phi_{2}}\leq b_{\Phi_{2}} <\infty$. Then  $\Phi_{1} \circ \Phi_{2} \in \mathscr{C}^{1}(\mathbb{R}_{+})$ growth function and  $$ a_{\Phi_{1}}a_{\Phi_{2}} \leq a_{\Phi_{1}\circ \Phi_{2}}\leq b_{\Phi_{1}\circ \Phi_{2}}\leq b_{\Phi_{1}}b_{\Phi_{2}}.$$
\end{proposition}

\begin{proof}
For $t>0$, we have  $$   \left(\Phi_{1} \circ \Phi_{2}\right)'(t)=\Phi_{1}'\left(\Phi_{2}(t) \right)\Phi_{2}'(t) \Rightarrow  \frac{t\left(\Phi_{1} \circ \Phi_{2}\right)'(t)}{\Phi_{1} \circ \Phi_{2}(t)} =\frac{\Phi_{2}(t)\Phi_{1}'\left(\Phi_{2}(t) \right)}{\Phi_{1}\left( \Phi_{2}(t) \right)} \times \frac{t\Phi_{2}'(t)}{\Phi_{2}(t)}.    $$
It follows that
$$ a_{\Phi_{1}}a_{\Phi_{2}} \leq \frac{t\left(\Phi_{1} \circ \Phi_{2}\right)'(t)}{\Phi_{1} \circ \Phi_{2}(t)} \leq  b_{\Phi_{1}}b_{\Phi_{2}},~~\forall~t>0. $$
\end{proof}

\begin{proposition}\label{pro:main4pm4}
Let  $\Phi \in \mathscr{C}^{1}(\mathbb{R}_{+})$ a growth function. The following assertions are equivalent.
\begin{itemize}
\item[(i)]  $0< a_\Phi\leq b_\Phi <\infty$.
\item[(ii)]  $0< a_{\Phi^{-1}}\leq b_{\Phi^{-1}} <\infty$. 
\end{itemize}
Moreover, $ a_{\Phi^{-1}}=1/b_{\Phi}\hspace*{0.5cm}\textrm{and}  \hspace*{0.5cm} b_{\Phi^{-1}}=1/a_{\Phi}.     $
\end{proposition}

\begin{proof}
Show that $i)$ implies $ii)$.
We have  $$   \left(\Phi^{-1}\right)'(t)=\frac{1}{\Phi'\left(\Phi^{-1}(t) \right)}, ~~\forall~t>0.      $$
It follows that
\begin{align*}
0< a_\Phi\leq b_\Phi <\infty &\Rightarrow 0< a_\Phi \leq \frac{t\Phi'(t)}{\Phi(t)} \leq b_\Phi <\infty, ~~\forall~t>0 \\
&\Rightarrow 0< a_\Phi \leq \frac{\Phi^{-1}(t)\Phi'\left(\Phi^{-1}(t) \right)}{\Phi\left(\Phi^{-1}(t)\right)} \leq b_\Phi <\infty, ~~\forall~t>0 \\
&\Rightarrow \frac{1}{b_\Phi} \leq \frac{t}{\Phi^{-1}(t)\Phi'\left(\Phi^{-1}(t) \right)} \leq \frac{1}{a_\Phi}, ~~\forall~t>0 \\
&\Rightarrow \frac{1}{b_\Phi} \leq \frac{t\left(\Phi^{-1}\right)'(t)}{\Phi^{-1}(t)} \leq \frac{1}{a_\Phi}, ~~\forall~t>0.
\end{align*}
We deduce on the one hand that
  \begin{equation}\label{eq:56aqi8n}
\frac{1}{b_{\Phi}} \leq a_{\Phi^{-1}} \leq  b_{\Phi^{-1}} \leq \frac{1}{a_{\Phi}}.\end{equation}
Reasoning as above, we obtain that (ii) implies (i) and we deduce on the other hand that
\begin{equation}\label{eq:56aqn}
\frac{1}{b_{\Phi^{-1}}} \leq a_{\Phi} \leq  b_{\Phi} \leq \frac{1}{a_{\Phi^{-1}}}.\end{equation}
From the Relations (\ref{eq:56aqi8n}) and (\ref{eq:56aqn}) we conclude that
 $a_{\Phi^{-1}}=1/b_{\Phi}$ and  $b_{\Phi^{-1}}=1/a_{\Phi}$.
\end{proof}

\begin{proposition}\label{pro:main4q4}
Let  $\Phi_{1}, \Phi_{2} \in \mathscr{L} \cup\mathscr{U}$. The following assertions are equivalent.
\begin{itemize}
\item[(i)] The function  $t\mapsto\frac{\Phi_{2}(t)}{\Phi_{1}(t)}$ is non-decreasing on $\mathbb{R}_{+}^{*}$.
\item[(ii)] The function  $t\mapsto\frac{\Phi_{2}\circ \Phi_{1}^{-1}(t)}{t}$ is non-decreasing on $\mathbb{R}_{+}^{*}$.
\item[(iii)] The function $\Phi_{2}\circ \Phi_{1}^{-1}$ belongs  $\mathscr{U}^{b_{\Phi_{2}}/a_{\Phi_{1}}}$.
\end{itemize}
\end{proposition}

\begin{proof}
The equivalence between  (i) and (ii) is obvious.
That (iii) implies (ii) is also immediate.
\vskip .1cm
Let us now show that (ii) implies (iii).
\vskip .1cm
Since the functions  $t\mapsto \frac{\Phi_{1}^{-1}(t)}{t^{1/a_{\Phi_{1}}}}$ and  $t\mapsto \frac{\Phi_{2}(t)}{t^{b_{\Phi_{2}}}}$ are non-increasing on $\mathbb{R}_{+}^{*}$, we deduce that for all
 $s>0$ and $t \geq 1$
$$  \Phi_{1}^{-1}(st) \leq t^{1/a_{\Phi_{1}}} \Phi_{1}^{-1}(s)  $$ 
and 
$$ \Phi_{2}\left(  t^{1/a_{\Phi_{1}}} \Phi_{1}^{-1}(s)  \right)  \leq t^{b_{\Phi_{2}}/a_{\Phi_{1}}}\Phi_{2}\left( \Phi_{1}^{-1}(s)  \right).  $$
 It follows that $$ \Phi_{2}\left(  \Phi_{1}^{-1}(st)  \right)   \leq  t^{b_{\Phi_{2}}/a_{\Phi_{1}}}\Phi_{2}\left( \Phi_{1}^{-1}(s)  \right). $$
\end{proof}

\begin{proposition}\label{pro:main44}
Let  $\Phi$ be a growth function such that  $\Phi(t)>0$  for all $t>0$. Consider $\widetilde{\Omega}$ the function defined by
$$ \widetilde{\Omega}(t)=\frac{1}{ \Phi\left(\frac{1}{t}\right)}, ~\forall~t>0 \hspace*{0.5cm}\textrm{and} \hspace*{0.5cm}  \widetilde{\Omega}(0)=0.      $$
The following assertions are satisfied. 
\begin{itemize}
\item[(i)]  $\Phi \in \mathscr{U}^{q}$ (resp.  $\mathscr{L}_{p}$) if and only if  $\widetilde{\Omega} \in \mathscr{U}^{q}$  (resp.  $\mathscr{L}_{p}$).
\item[(ii)]  $\Phi \in \mathscr{U}\cap \nabla_{2}$ if and only if  $\widetilde{\Omega} \in \mathscr{U}\cap \nabla_{2}$.
\end{itemize}
\end{proposition}

\begin{proof}
$i)$ Suppose that  $\Phi \in \mathscr{U}^{q}$.
For $0<t_{1}\leq t_{2}$, we have 
$$ \frac{\Phi(t_{1})}{t_{1}} \leq \frac{\Phi(t_{2})}{t_{2}} 
\Leftrightarrow  \frac{\Phi\left(1/t_{2}\right)}{1/t_{2}} \leq \frac{\Phi\left(1/t_{1}\right)}{1/t_{1}} 
\Leftrightarrow \frac{1}{t_{1}}\frac{1}{\Phi\left(1/t_{1}\right)} \leq \frac{1}{t_{2}}\frac{1}{\Phi\left(1/t_{2}\right)}
\Leftrightarrow \frac{\widetilde{\Omega}(t_{1})}{t_{1}} \leq \frac{\widetilde{\Omega}(t_{2})}{t_{2}}.
     $$
Since $\Phi$ is of upper type $q$ then so is the function $\widetilde{\Omega}$. Indeed, for all $s> 0$ and $t\geq 1$
$$ \Phi\left(\frac{1}{s}\right)=\Phi\left(t \times\frac{1}{st}\right)\leq C_{q}t^{q}\Phi\left(\frac{1}{st}\right) \Rightarrow
\frac{1}{C_{q}t^{q}\Phi\left(\frac{1}{st}\right)} \leq  \frac{1}{\Phi\left(\frac{1}{s}\right)} 
\Rightarrow  \widetilde{\Omega}(st) \leq C_{q}t^{q}\widetilde{\Omega}(s).
      $$
The converse is obtained similarly. We conclude that  $\Phi \in \mathscr{U}^{q}$ if and only if $\widetilde{\Omega} \in \mathscr{U}^{q}$.\\
Reasoning in the same way, we also show that  $\Phi \in \mathscr{L}_{p}$ if and only if $\widetilde{\Omega} \in \mathscr{L}_{p}$.\\
(ii) We suppose that  $\Phi \in \mathscr{U}\cap \nabla_{2}$. 
For $t>0$, we have
$$   \Phi\left( \frac{1}{t}\right) \leq\frac{1}{2C} \Phi\left( \frac{C}{t}\right) \Rightarrow  \frac{2C}{\Phi\left( \frac{C}{t}\right)} \leq \frac{1}{\Phi\left( \frac{1}{t}\right)} \Rightarrow 2C\widetilde{\Omega}\left( \frac{t}{C}\right) \leq \widetilde{\Omega}(t), $$
according to the Lemma \ref{pro:main40l6aqlj}. We deduce that  $\widetilde{\Omega} \in \mathscr{U}\cap \nabla_{2}$.\\
The converse is obtained similarly.
\end{proof}

\begin{lemme}\label{pro:main40}
Let  $\Phi_{1}, \Phi_{2} \in \mathscr{L} \cup\mathscr{U}$ and put 
$$  \widetilde{\Omega}_{3}(t)=\frac{1}{\Phi_{2}\circ \Phi_{1}^{-1}\left(\frac{1}{t}\right)}, ~\forall~t>0 \hspace*{0.5cm}\textrm{and} \hspace*{0.5cm}  \widetilde{\Omega}_{3}(0)=0.    $$
If the function  $t\mapsto\frac{\Phi_{2}(t)}{\Phi_{1}(t)}$ is non-decreasing on $\mathbb{R}_{+}^{*}$ then  $\widetilde{\Omega}_{3} \in \mathscr{U}$.
\end{lemme}

\begin{proof}
The proof  follows from Proposition \ref{pro:main4q4} and Proposition \ref{pro:main44}.

\end{proof}

\begin{lemme}\label{pro:main80jdf}
Let $\Phi\in \widetilde{\mathscr{L}}\cup\widetilde{\mathscr{U}}$. There exists a constant $C >0$ such that 
\begin{equation}\label{eq:ibetdfide1}
\Phi\left(\frac{s}{t}\right)\leq C\frac{\Phi(s)}{\Phi(t)}, ~~ \forall~s,t >0 
.\end{equation}
\end{lemme}

\begin{proof}
The inequality (\ref{eq:ibetdfide1}) is true for $\Phi\in \widetilde{\mathscr{U}}$ (see. \cite[Lemma 4.3]{djesehaqb}).\\
For  $0<p \leq 1$ suppose that $\Phi\in \widetilde{\mathscr{L}}_{p}$.
For  $s,t >0$, we have
$$  \Phi\left(\frac{s}{t}\right)\leq  C_{1}\Phi(s) \Phi\left(\frac{1}{t}\right),   $$
since the inequality (\ref{eq:ibejecaqti2}) is satisfied.\\
If $0<t<1$ then  we have
$$  \Phi(t)= \Phi\left(\frac{1}{\frac{1}{t} } \right) \leq C_{2} \frac{1^{p}}{\Phi\left(\frac{1}{t}\right)},            $$
thanks to Relation (\ref{eq:iaqide2}). It follows that
\begin{equation}\label{eq:ibetdfaqide1}
\Phi\left(\frac{1}{t}\right) \leq C_{2}\frac{1}{\Phi(t)}. 
\end{equation}
If $t \geq 1$ then  we have
$$  \Phi\left(\frac{1}{t}\right)=\Phi\left(\frac{1}{t}\times 1\right) \leq C_{p} \left(\frac{1}{t}\right)^{p}\Phi(1),               $$
since  $\Phi$ is of lower type $p$. 
It follows that
\begin{equation}\label{eq:ibetdaqfaqide1}
\Phi\left(\frac{1}{t}\right) \leq \frac{C_{2}}{\Phi(1)}\frac{1}{\Phi(t)}, 
\end{equation}
since from Relation (\ref{eq:iaqide2}), we have also
$$  \Phi(t)=\Phi\left( \frac{t}{1}\right) \leq C_{2} \frac{t^{p}}{\Phi(1)}.$$
From Relations (\ref{eq:ibetdfaqide1}) and (\ref{eq:ibetdaqfaqide1}), we deduce that
$$ \Phi\left(\frac{1}{t}\right) \lesssim \frac{1}{\Phi(t)}. 
    $$
Therefore, 
$$ \Phi\left(\frac{s}{t}\right)\lesssim  \frac{\Phi(s)}{\Phi(t)}.   $$
\end{proof}

\subsection{Some properties of Orlicz spaces.}

Let $(X, \sum, \mu)$ be a measure space and $\Phi$ a  growth function. The Orlicz space on  $X$, $L^{\Phi}(X, d\mu)$ is the set of all equivalent
classes (in the usual sense) of measurable functions $f : X \longrightarrow \mathbb{C}$ which satisfy
$$ \|f\|_{L_{\mu}^{\Phi}}^{lux}:=\inf\left\{\lambda>0 : \int_{X}\Phi\left(\dfrac{| f(x)|}{\lambda}\right)d\mu(x) \leq 1  \right\}< \infty. $$
If $\Phi$ is convex  then $(L^{\Phi}(X, d\mu),  \|.\|_{L_{\mu}^{\Phi}}^{lux})$  is a Banach space (see.\cite{shuchen, kokokrbec, raoren}).  The space  $L^{\Phi}$  generalizes the Lebesgue space $L^{p}$  for $0< p < \infty$.

Let $\Phi$ be a growth function. Let  $f\in L^{\Phi}(X, d\mu)$ and  put 
$$ \|f\|_{L_{\mu}^{\Phi}}:= \int_{X}\Phi\left(| f(x)|\right)d\mu(x).   $$
If  $\Phi \in \mathscr{C}^{1}(\mathbb{R}_{+})$ is a growth function such that   $0< a_\Phi\leq b_\Phi <\infty$, then we have the following inequalities
$$  \|f\|_{L_{\mu}^{\Phi}} \lesssim \max\left\{ \left(  \|f\|_{L_{\mu}^{\Phi}}^{lux} \right)^{a_\Phi}; \left( \|f\|_{L_{\mu}^{\Phi}}^{lux}\right)^{b_\Phi}\right\}     $$
and
$$   \|f\|_{L_{\mu}^{\Phi}}^{lux} \lesssim \max\left\{ \left( \|f\|_{L_{\mu}^{\Phi}}\right)^{1/a_\Phi} ;  \left( \|f\|_{L_{\mu}^{\Phi}}  \right)^{1/b_\Phi}\right\}.
    $$
We will simply denote  $L^{\Phi}(\mathbb{R})= L^{\Phi}(\mathbb{R}, dx)$, where $dx$ is the Lebesgue measure on $\mathbb{R}$.

Let $\Phi$ be a convex growth function. We have the following inclusion
$$ L^{\Phi}(\mathbb{R})\subset  L^{1}\left(\mathbb{R},\frac{dt}{1+t^{2}}\right)$$

Let $\alpha>-1$ and $E$ be a measurable set of $\mathbb{C}_{+}$. We denote 
$$ |E|_{\alpha}:=\int_{E}dV_{\alpha}(x+iy).    $$

Let $I$ be an interval and $Q_{I}$ its associated Carleson square. It is easy to see that
\begin{equation}\label{eq:fo4n1l}
 |Q_{I}|_{\alpha}=\frac{1}{1+\alpha}|I|^{2+\alpha}. \end{equation}

Fix $\beta \in \left\{0; 1/3\right\}$. An interval $\beta-$dyadic is any interval $I$ of $\mathbb{R}$ of the form
 $$   2^{-j}(\lbrack 0,1) +k+(-1)^{j}\beta),    $$
 where $k,j\in \mathbb{Z}$. We denote by $\mathcal{D}_{j}^{\beta}$ the set of $\beta-$dyadic intervals $I$ such that $|I|=2^{-j}$.  Put
 $ \mathcal{D}^{\beta}:=\bigcup_{j}\mathcal{D}_{j}^{\beta}$.\\
 We have the following properties (see for example \cite{DVCruzJMMartell,Emstshae}):
 \begin{itemize}
 \item[-] for all $I,J\in \mathcal{D}^{\beta}$, we have $I\cap J \in \left\{\emptyset ; I; J\right\}$,
 \item[-] for each fixed  $j\in \mathbb{Z}$, if $I \in \mathcal{D}_{j}^{\beta}$ then there exists a unique $J \in \mathcal{D}_{j-1}^{\beta}$ such that  $I\subset J$,
 \item[-] for each fixed $j\in \mathbb{Z}$, if $I \in \mathcal{D}_{j}^{\beta}$ then there exists $I_{1}, I_{2} \in \mathcal{D}_{j+1}^{\beta}$ such that
  $I=I_{1}\cup I_{2}$ and $I_{1}\cap I_{2}=\emptyset$.
 \end{itemize}
%\cite{Spmcregu}
We refer to \cite{hytperez,Spmcregu} for the following.
 
 \begin{lemme}\label{pro:main80}
 Let $I$ be an interval.  There exists    $\beta \in \left\{0, 1/3\right\}$ and $J \in \mathcal{D}^{\beta}$ such that $I\subset J $ and $|J|\leq 6|I|$.
 \end{lemme}

Let  $\alpha>-1$ and $f$  a measurable function on $\mathbb{R}$ (resp. $\mathbb{C}_{+}$). The Hardy-Littlewood maximal functions on the line and on the upper-half plane for a function of $f$ are respectively defined by
 $$  \mathcal{M}_{HL}(f)(x):=\sup_{I \subset \mathbb{R} }\frac{\chi_{I}(x)}{|I|}\int_{I}| f(t)| dt, ~~\forall~ x \in \mathbb{R},      $$
and  
$$ \mathcal{M}_{V_{\alpha}}(f)(z):= \sup_{I \subset \mathbb{R}}\frac{\chi_{Q_{I}}(z)}{|Q_{I}|_{\alpha}}\int_{Q_{I}}| f(\omega)| dV_{\alpha}(\omega),~~ \forall~ z \in \mathbb{C}_{+},     $$ 
where the supremum is taken over all intervals of $\mathbb{R}$.  Similarly, for $\beta \in \left\{0; 1/3\right\}$, we define their dyadic versions $\mathcal{M}^{\mathcal{D}^{\beta}}_{HL}(f)$ and  $\mathcal{M}_{V_{\alpha}}^{\mathcal{D}^{\beta}}(f)$ as above but with the supremum taken this time on the intervals in the dyadic grid $\mathcal{D}^{\beta}$. We have
  \begin{equation}\label{eq:of895ua1l}
   \mathcal{M}_{HL}(f) \leq 6 \sum_{\beta \in \{0; 1/3\} }   \mathcal{M}^{\mathcal{D}^{\beta}}_{HL}(f)
  \end{equation}
 and 
 \begin{equation}\label{eq:ooua5x1l}
\mathcal{M}_{V_{\alpha}} (f) \leq 6^{2+\alpha} \sum_{\beta \in \{0;1/3\}}   \mathcal{M}^{\mathcal{D}^{\beta}}_{V_{\alpha}} (f).
\end{equation}

\begin{proposition}\label{pro:main9aaq9}
Let  $\beta \in \left\{0; 1/3\right\}$,  $\alpha>-1$,  $0<\gamma<\infty$ and $\Phi$ a growth function. Put
 $$ \Phi_{\gamma}(t):=\Phi(t^{\gamma}), ~~\forall~t\geq 0.    $$
If  $\Phi_{\gamma}$ is convex then the following assertions are satisfied
 \begin{itemize}
  \item[(i)] for all  $0\not\equiv f\in L^{\Phi}(\mathbb{R})$ and for   $\lambda > 0$, 
  $$ \left| \left\{x\in \mathbb{R} : \left(\mathcal{M}^{\mathcal{D}^{\beta}}_{HL}\left(\left(\frac{|f|}{\|f\|_{L^{\Phi}}^{lux}}\right)^{1/\gamma}\right)(x)\right)^{\gamma}> \lambda \right\}    \right| \leq \frac{1}{\Phi(\lambda)}.    $$
\item[(ii)] for all  $0\not\equiv f\in L^{\Phi}(\mathbb{C_{+}}, dV_{\alpha})$ and for  $\lambda > 0$
$$  \left| \left\{z\in \mathbb{C}_{+} : \left(\mathcal{M}^{\mathcal{D}^{\beta}}_{V_{\alpha}}\left(\left(\frac{|f|}{\|f\|_{L_{\alpha}^{\Phi}}^{lux}}\right)^{1/\gamma}\right)(z)\right)^{\gamma}> \lambda \right\}    \right|_{\alpha} \leq \frac{1}{\Phi(\lambda)}.    $$
  \end{itemize}
\end{proposition}

\begin{proof}
$i)$ Let $0\not\equiv f\in L^{\Phi}(\mathbb{R})$ and   put 
 $$   g:= \frac{|f|^{1/\gamma}}{\left(\|f\|_{L^{\Phi}}^{lux}\right)^{1/\gamma}}.     $$
We have $$ \int_{\mathbb{R}}\Phi_{\gamma}(|g(x)|)dx= \int_{\mathbb{R}}\Phi_{\gamma}\left(\left( \frac{|f(x)|}{\|f\|_{L^{\Phi}}^{lux}}\right)^{1/\gamma}\right)dx = \int_{\mathbb{R}}\Phi\left( \frac{|f(x)|}{\|f\|_{L^{\Phi}}^{lux}}\right)dx \leq 1.          $$
We deduce that $g \in L^{\Phi_{\gamma}}(\mathbb{R})$ and
 $\|g\|_{L^{\Phi_{\gamma}}}^{lux} \leq 1$.\\
For $\lambda > 0$, we can therefore find $\{I_{j}\}_{j\in \mathbb{N}}$ a family of pairwise disjoint $\beta-$dyadic intervals such that
$$\left\{x\in \mathbb{R} : \mathcal{M}^{\mathcal{D}^{\beta}}_{HL}(g)(x)> \lambda^{1/\gamma} \right\} = \bigcup_{j\in \mathbb{N}}I_{j},  $$
and 
$$\lambda^{1/\gamma} <  \frac{1}{|I_{j}|}\int_{I_{j}}|g(y)| dy, ~~ \forall~j\in \mathbb{N}.  $$
For $j \in \mathbb{N}$,  we have
$$  \Phi(\lambda)=\Phi_{\gamma}\left(\lambda^{1/\gamma} \right) \leq  \Phi_{\gamma}\left( \frac{1}{|I_{j}|}\int_{I_{j}}|g(y)| dy  \right) \leq \frac{1}{|I_{j}|}\int_{I_{j}}\Phi_{\gamma}(|g(y)|) dy,       $$
thanks to Jensen's inequality. We deduce that
$$ |I_{j}| \leq  \frac{1}{\Phi(\lambda)}\int_{I_{j}}\Phi_{\gamma}(|g(y)|) dy, ~~ \forall~j\in \mathbb{N}.            $$
It follows that 
\begin{align*}
\left|\left\{x\in \mathbb{R} : \mathcal{M}^{\mathcal{D}^{\beta}}_{HL}(g)(x)> \lambda^{1/\gamma} \right\}\right|&= \sum_{j}|I_{j}| \\
&\leq \sum_{j}\frac{1}{\Phi(\lambda)}\int_{I_{j}}\Phi_{\gamma}(|g(y)|) dy \\
&= \frac{1}{\Phi(\lambda)}\int_{\bigcup_{j}I_{j}}\Phi_{\gamma}(|g(y)|) dy 
\leq \frac{1}{\Phi(\lambda)}.
\end{align*}
In the same way, we prove the inequality of the point (ii).
 \end{proof}

\begin{theorem}\label{pro:mainplaaqq6}
Let   $\alpha>-1$ and $\Phi_{1}, \Phi_{2} \in \mathscr{U}$. The following assertions are equivalent.
 \begin{itemize}
\item[(i)] There exists a constant $C_{1}>0$ such that for all $t>0$, 
\begin{equation}\label{eq:conditdedinis}
\int_0^t\frac{\Phi_{2}(s)}{s^{2}}ds\le C_{1}\frac{\Phi_{1}(t)}{t}.\end{equation}
\item[(ii)] There exists a constant $C_{2}>0$ such that for all $f\in L^{\Phi_{1}}(\mathbb{R})$, 
\begin{equation}\label{eq:foncinaqimal}
\|\mathcal{M}_{HL}(f)\|_{L^{\Phi_{2}}}^{lux} \leq C_{2} \|f\|_{L^{\Phi_{1}}}^{lux}.
\end{equation}
\item[(iii)] There exists a constant $C_{3}>0$ such that for all $f\in L^{\Phi_{1}}(\mathbb{C_{+}}, dV_{\alpha})$,
\begin{equation}\label{eq:foncinimal}
\|\mathcal{M}_{V_{\alpha}}(f)\|_{L_{V_{\alpha}}^{\Phi_{2}}}^{lux} \leq C_{3} \|f\|_{L_{V_{\alpha}}^{\Phi_{1}}}^{lux}.
\end{equation}
 \end{itemize}
\end{theorem}

\begin{proof}
$i) \Leftrightarrow ii)$ This equivalence follows from the \cite[Lemma 3.15]{djesehaqbAZ}. \\
$(i)\Rightarrow (iii)$ The proof of this implication is identical to that of the \cite[Proposition 3.12]{djesehb}. \\
\medskip

Let us show that (iii) implies (i).  Assume that  inequality (\ref{eq:conditdedinis}) is not satisfied.
We can find a sequence of positive reals $(t_{k})_{k \geq 1}$ such that 
\begin{equation}\label{eq:cotdinis}
\int_{0}^{t_{k}}\frac{\Phi_{2}(s)}{s^2}ds\geq \dfrac{2^{k}\Phi_{1}(2^{k}t_{k})}{t_{k}},~~ \forall~k \geq 1 .\end{equation}
For  $k \geq 1$, put
$$    f_{k}:=2^{k}t_{k}\chi_{Q_{I_{k}}},      $$
where  $Q_{I_{k}}$ is the Carleson square associated with the interval $I_{k}$ given as follows:  $$I_{k}:=\left\{x\in \mathbb{R} : \sum_{j=0}^{k-1}  \left(\frac{\alpha+1}{2^{j}\Phi_{1}(2^{j}t_{j})}\right)^{\frac{1}{\alpha+2}}\leq x<\sum_{j=0}^{k}  \left(\frac{\alpha+1}{2^{j}\Phi_{1}(2^{j}t_{j})}\right)^{\frac{1}{\alpha+2}} \right\}$$ 
From the relation (\ref{eq:fo4n1l}), we have
$$ |Q_{I_{k}}|_{\alpha} =\frac{1}{1+\alpha}|I_{k}|^{2+\alpha}= \frac{1}{2^{k}\Phi_{1}(2^{k}t_{k})}. $$ 
It follows that  $f_{k}\in L^{\Phi_{1}}(\mathbb{C}_{+},  dV_{\alpha})$. In indeed 
$$  \int_{\mathbb{C}_{+}}\Phi_{1}(|f_{k}(z)|)dV_{\alpha}(z)=\int_{Q_{I_{k}}}\Phi_{1}(2^{k}t_{k})dV_{\alpha}(z)=\Phi_{1}(2^{k}t_{k})|Q_{I_{k}}|_{\alpha}
=  \dfrac{1}{2^{k}}< \infty.       $$
According the Lemma \ref{pro:main80}, there exists a dyadic interval $J_{k}\in \mathcal{D}^{\beta}$ such that $I_{k}\subset J_{ k}$ and $|J_{k}|\leq 6|I_{k}|$.
Let $z\in Q_{I_{k}}$. We have
$$ |f_{k}(z)|= \frac{1}{|Q_{I_{k}}|_{\alpha}}\int_{Q_{I_{k}}}2^{k}t_{k}\chi_{Q_{I_{k}}}(\omega)dV_{\alpha}(\omega) \leq 6^{2+\alpha}\frac{\chi_{Q_{J_{k}}}(z)}{|Q_{J_{k}}|_{\alpha}}\int_{Q_{J_{k}}}|f_{k}(\omega)|dV_{\alpha}(\omega),    $$
where  $Q_{J_{k}}$ is the Carleson square associated with $J_{k}$. We deduce that  
$$ |f_{k}(z)| \leq 6^{2+\alpha}\mathcal{M}_{V_{\alpha}}^{\mathcal{D}^{\beta}}(f_{k})(z), ~~ 
\forall~ z \in \mathbb{C}_{+}.
       $$ 
It follows that for $\lambda >0$, 
\begin{equation}\label{eq:cotaqdinis}
\frac{1}{\lambda}\int_{\{z\in \mathbb{C}_{+}~ :~ |f_{k}(z)|> \lambda \}}|f_{k}(z)|dV_{\alpha}(z) \leq 2^{2+\alpha}\left|\left\{z\in \mathbb{C_{+}} : \mathcal{M}_{V_{\alpha}}^{\mathcal{D}^{\beta}}(6^{2+\alpha}f_{k})(z)> \lambda \right\}\right|_{\alpha}.\end{equation}
Put
$$  f(z)= \sum_{k=1}^{\infty}6^{2+\alpha}f_{k}(z), ~~\forall~ z \in \cup_{k\geq 1}Q_{I_{k}} \hspace*{0.5cm}\textrm{and} \hspace*{0.5cm}  f(z)=0, ~~\forall~ z \in \mathbb{C}_{+} \backslash \cup_{k\geq 1}Q_{I_{k}}.     $$
Since the $I_{k}$ are pairwise disjoint, the same is true for the $Q_{I_{k}}$. So we have
$$  \int_{\mathbb{C}_{+}}\Phi_{1}(|f(z)|)dV_{\alpha}(z) \lesssim \sum_{k=1}^{\infty}\int_{\mathbb{C}_{+}}\Phi_{1}(|f_{k}(z)|)dV_{\alpha}(z)
=\sum_{k=1}^{\infty}\int_{Q_{I_{k}}}\Phi_{1}(2^{k}t_{k})\chi_{Q_{I_{k}}}(z)dV_{\alpha}(z)
= \sum_{k=1}^{\infty} \dfrac{1}{2^{k}}< \infty.       $$
We deduce that $f\in L^{\Phi_{1}}(\mathbb{C}_{+},  dV_{\alpha})$. \\
Since the inequalities (\ref{eq:cotdinis}) and (\ref{eq:cotaqdinis}) are satisfied, we have
\begin{align*}
\int_{\mathbb{C}_{+}}\Phi_{2}\left( \mathcal{M}_{V_{\alpha}}(f)(z)\right)dV_{\alpha}(z) &\gtrsim
\int_{0}^{\infty}\Phi_{2}^{\prime}(\lambda)\left|\left\{z\in \mathbb{C}_{+} : \mathcal{M}_{V_{\alpha}}^{\mathcal{D}^{\beta}}(6^{2+\alpha}f_{k})(z)> \lambda \right\}\right|_{\alpha}d\lambda\\
&\gtrsim \int_{0}^{\infty}\Phi_{2}^{\prime}(\lambda) \left( \frac{1}{\lambda}\int_{\{\omega\in \mathbb{C}_{+}~ :~ |f_{k}(\omega)|> \lambda \}}|f_{k}(z)|dV_{\alpha}(z) 
 \right)d\lambda\\
&\gtrsim \int_{\mathbb{C}_{+}}|f_{k}(z)|\left(\int_{0}^{|f_{k}(z)|} \frac{\Phi_{2}(\lambda)}{\lambda^{2}}d\lambda \right)dV_{\alpha}(z)\\
&\gtrsim 2^{k}t_{k}|Q_{I_{k}}|_{\alpha}\left(\int_{0}^{2^{k}t_{k}} \frac{\Phi_{2}(\lambda)}{\lambda^{2}}d\lambda \right)\\
&\gtrsim 2^{k}.
\end{align*}
We deduce that $\mathcal{M}_{V_{\alpha}}(f) \not\in L^{\Phi_{2}}(\mathbb{C}_{+},  dV_{\alpha})$.
\end{proof}

\begin{corollaire}\label{pro:mainplaqaq6}
Let   $\alpha>-1$ and $\Phi \in \mathscr{U}$. The following assertions are equivalent.
\begin{itemize}
\item[(i)] $\Phi \in \nabla_{2}$.
\item[(ii)] $\mathcal{M}_{HL}: L^{\Phi}(\mathbb{R})\longrightarrow L^{\Phi}(\mathbb{R})$ is bounded.
\item[(iii)] $\mathcal{M}_{V_{\alpha}}: L^{\Phi}(\mathbb{C}_{+}, dV_{\alpha})\longrightarrow L^{\Phi}(\mathbb{C}_{+}, dV_{\alpha})$ is bounded.
\end{itemize}
\end{corollaire}

\subsection{Some properties of Hardy-Orlicz and Bergman-Orlicz spaces on $\mathbb{C}_{+}$.}

Let  $\Phi$ be a growth function and $F\in H^{\Phi}(\mathbb{C}_{+})$. Put 
$$  \|F\|_{H^{\Phi}}:=\sup_{y>0}\int_{\mathbb{R}}\Phi\left(| F(x+iy)|\right)dx.   $$
Let $\Phi \in \mathscr{C}^{1}(\mathbb{R}_{+})$ a growth function such that   $0< a_\Phi\leq b_\Phi <\infty$. We have the following inequalities
$$  \|F\|_{H^{\Phi}} \lesssim \max\left\{ \left(  \|F\|_{H^{\Phi}}^{lux} \right)^{a_\Phi}; \left( \|F\|_{H^{\Phi}}^{lux}\right)^{b_\Phi}\right\}   $$
and
$$  \|F\|_{H^{\Phi}}^{lux} \lesssim \max\left\{ \left( \|F\|_{H^{\Phi}}\right)^{1/a_\Phi} ;  \left( \|F\|_{H^{\Phi}}  \right)^{1/b_\Phi}\right\}.      $$

Let $\Omega$ be an open set of $\mathbb{C}$ and $F: \Omega \longrightarrow ]-\infty, +\infty]$ a function. We say that $F$ is subharmonic if the following assertions are satisfied:
\begin{itemize}
  \item[(i)] $F$ is upper semicontinuous on $\Omega$ $$ F(z_{0}) \geq \lim_{z\to z_{0}} F(z), ~~\forall~z_{0}\in \Omega,     $$  
 \item[(ii)]  for all $z_{0}\in \Omega$, there exists $r(z_{0})>0$ such that $\mathcal{D}(z_{0}, r(z_{0}))=\{z\in \Omega: |z-z_{0}| < r(z_{0})\}$ is contained in $\Omega$ and such that for all $r < r(z_{0})$
 \begin{equation}\label{eq:ualp5lepmn}
  F(z_{0})\leq \frac{1}{\pi r^{2}}\int \int_{|x+iy-z_{0}| < r}F(x+iy)dxdy
  .\end{equation}
\end{itemize} 

\begin{proposition}\label{pro:main6mq4}
Let $\Phi$ be a growth function such that  $\Phi(t)>0$ for all  $t>0$. If $\Phi$ is convex or belongs to $\mathscr{L}$ then for  $F \in H^{\Phi}(\mathbb{C}_{+})$, we have
 \begin{equation}\label{eq:ualp5leson}
 |F(x+iy)|\leq \Phi^{-1}\left(\frac{2}{\pi y}\right)\|F\|_{H^{\Phi}}^{lux}, ~~ \forall~x+iy \in \mathbb{C_{+}}.\end{equation}
 \end{proposition}

\begin{proof}
For $t\geq 0$, put
 $$  \Phi_{\rho}(t)=\Phi\left(t^{1/\rho}\right),      $$
 where $\rho=1$ if $\Phi$ is convex and $\rho=a_{\Phi}$ if  $\Phi \in \mathscr{L}$.
By construction, $\Phi_{\rho}$ is a convex growth function. 
Let $0\not\equiv F \in H^{\Phi}(\mathbb{C}_{+})$, and $z_{0}=x_{0}+iy_{0}\in \mathbb{C_{+}}$ and  $r=\dfrac{y_{0}}{2}$. Since  $|F|^{\rho}$ is subharmonic on $\mathbb{C}_ {+}$, we have
$$ |F(z_{0})|^{\rho} \leq \frac{1}{\pi r^{2}}\int \int_{\overline{\mathcal{D}(z_{0}, r)}}|F(u+iv)|^{\rho}dudv,   $$
where $\mathcal{D}(z_{0}, r)$ is the disk centered at $z_{0}$ and of radius $r$. By  Jensen's inequality, it follows that
\begin{align*}
 \Phi\left(\frac{|F(z_{0})|}{\|F\|_{H^{\Phi}}^{lux}}\right) &\leq \Phi_{\rho}\left( \frac{1}{\pi r^{2}}\int \int_{\overline{\mathcal{D}(z_{0}, r)}}\left(\frac{|F(u+iv)|}{\|F\|_{H^{\Phi}}^{lux}}\right)^{\rho}dudv \right)\\
 &\leq\frac{1}{\pi r^{2}}\int \int_{\overline{\mathcal{D}(z_{0}, r)}}\Phi\left(\frac{|F(u+iv)|}{\|F\|_{H^{\Phi}}^{lux}}\right)dudv \\
 &\leq \frac{1}{\pi r^{2}}\int_{0}^{2r}\int_{\mathbb{R}}\Phi\left(\dfrac{|F(u+iv)|}{\|F\|_{H^{\Phi}}^{lux}}\right)du dv  
 \leq  \frac{1}{\pi r^{2}}\int_{0}^{2r}dv.
 \end{align*}
We deduce that $$ \Phi\left(\frac{|F(z_{0})|}{\|F\|_{H^{\Phi}}^{lux}}\right) \leq  \frac{2}{\pi r}, ~~\forall~r<y_{0}.       $$
 \end{proof}

 \begin{lemme}\label{pro:main6apaqmq4}
Let $\Phi$ be a growth function such that  $\Phi(t)>0$ for all  $t>0$. If $\Phi$ is convex or belongs to $\mathscr{L}$, then  for  $F \in H^{\Phi}(\mathbb{C}_{+})$ and for  $\beta>0$, we have 
\begin{equation}\label{eq:5aaqq6}
  \Phi(|F(z+i\beta)|) \leq  \frac{1}{\pi}\int_{\mathbb{R}}\frac{y}{(x-t)^{2}+y^{2}}\Phi(|F(t+i\beta)|)dt,~~\forall~ z=x+iy \in \mathbb{C_{+}}.
 \end{equation} 
\end{lemme}
 
 \begin{proof}
For $t\geq 0$, put
 $$  \Phi_{\rho}(t)=\Phi\left(t^{1/\rho}\right),      $$
 where $\rho=1$ if $\Phi$ is convex and $\rho=a_{\Phi}$ if  $\Phi \in \mathscr{L}$.\\
Let $0\not\equiv F \in H^{\Phi}(\mathbb{C}_{+})$ and  $\beta>0$. For $z\in \mathbb{C_{+}}$, put 
 $$  U_{\beta}(z) = |F(z+i\beta)|^{\rho}.         $$
By construction,  $U_{\beta}$ is continuous on $\overline{\mathbb{C_{+}}} := \mathbb{C_{+}} \cup \mathbb{R}$ and subharmonic on $\mathbb{C_{+}}$. For $z=x+iy \in \overline{\mathbb{C_{+}}}$, we have
$$|U_{\beta}(z)|= |F(x+i(y+\beta))|^{\rho}\leq \left(  \Phi^{-1}\left(\frac{2}{\pi (y+\beta)}\right)\|F\|_{H^{\Phi}}^{lux} \right)^{\rho}
 \leq \left(\Phi^{-1}\left(\frac{2}{\pi \beta}\right)\|F\|_{H^{\Phi}}^{lux} \right)^{\rho},      $$
 according to  Proposition \ref{pro:main6mq4}. We deduce that $U_{\beta}$ is bounded on $\overline{\mathbb{C_{+}}}$. It follows that 
$$|F(z+i\beta)|^{\rho} \leq\frac{1}{\pi}\int_{\mathbb{R}}\frac{y}{(x-t)^{2}+y^{2}}|F(t+i\beta)|^{\rho}dt,~~\forall~z= x+iy \in \mathbb{C_{+}}, $$
thanks to  \cite[Corollary 10.15]{javadmas}. Since $\Phi_{\rho}$ is convex, by Jensen's inequality we deduce that
$$  \Phi(|F(z+i\beta)|) \leq  \frac{1}{\pi}\int_{\mathbb{R}}\frac{y}{(x-t)^{2}+y^{2}}\Phi(|F(t+i\beta)|)dt,~~\forall~z= x+iy \in \mathbb{C_{+}}.      $$
\end{proof}

 \begin{proposition}\label{pro:main6apmqaaqq4}
  Let $\Phi$ be a growth function such that  $\Phi(t)>0$ for all  $t>0$ and $F$ an analytic function on $\mathbb{C}_{+}$. If $\Phi$ is convex or belongs to $\mathscr{L}$, then  the following assertions are equivalent.
  \begin{itemize}
   \item[(i)]  $F\in H^{\Phi}(\mathbb{C}_{+})$. 
   \item[(ii)] The function $y\mapsto \|F(.+iy)\|_{L^{\Phi}}^{lux}$ is non-increasing on $\mathbb{R_{+}^{*}}$ and $\lim_{y \to 0}\|F(.+iy)\|_{L^{\Phi}}^{lux}< \infty.$
  \end{itemize}
  Moreover, 
  $$ \|F\|_{H^{\Phi}}^{lux} =   \lim_{y \to 0}\|F(.+iy)\|_{L^{\Phi}}^{lux}.         $$ 
   \end{proposition}

\begin{proof}
The implication $(ii) \Rightarrow (i)$ is immediate.\\ Let us now show that (i) implies (ii).
Suppose that $F\not\equiv 0$ is non-identically zero because there is nothing to show when $F\equiv 0$. 
Let  $0<y_{1}<y_{2}$. According to Lemma \ref{pro:main6apaqmq4} and Fubbini's theorem, we have
\begin{align*}
\int_{\mathbb{R}}\Phi\left(\frac{|F(x+iy_{2})|}{\|F(.+iy_{1})\|_{L^{\Phi}}^{lux}}\right)dx
&=\int_{\mathbb{R}}\Phi\left(\frac{|F(x+i(y_{2}-y_{1})+iy_{1})|}{\|F(.+iy_{1})\|_{L^{\Phi}}^{lux}}\right)dx \\
&\leq \int_{\mathbb{R}}\frac{1}{\pi}\int_{\mathbb{R}}\frac{(y_{2}-y_{1})}{(x-t)^{2}+(y_{2}-y_{1})^{2}}\Phi\left(\frac{|F(t+iy_{1})|}{\|F(.+iy_{1})\|_{L^{\Phi}}^{lux}}\right)dt dx \\
&=\int_{\mathbb{R}}\Phi\left(\frac{|F(t+iy_{1})|}{\|F(.+iy_{1})\|_{L^{\Phi}}^{lux}}\right)\left(  \frac{1}{\pi}\int_{\mathbb{R}}\frac{(y_{2}-y_{1})}{(x-t)^{2}+(y_{2}-y_{1})^{2}} dx\right)dt\\
&= \int_{\mathbb{R}}\Phi\left(\frac{|F(t+iy_{1})|}{\|F(.+iy_{1})\|_{L^{\Phi}}^{lux}}\right) dt  
\leq 1.  
\end{align*}
We deduce that  $  \|F(.+iy_{2})\|_{L^{\Phi}}^{lux} \leq \|F(.+iy_{1})\|_{L^{\Phi}}^{lux}. $
Therefore,  $$ \sup_{y>0}\|F(.+iy)\|_{L^{\Phi}}^{lux}=    \lim_{y \to 0}\|F(.+iy)\|_{L^{\Phi}}^{lux}.     $$
\end{proof}

Let $\Phi$ be a  growth function. The Hardy space on $\mathbb{D}$, $H^{\Phi}(\mathbb{D})$ is the set of analytic function $G$ on  $\mathbb{D}$ which satisfy
$$ \|G\|_{H^{\Phi}(\mathbb{D})}^{lux}:=\sup_{0\leq r<1 }\inf\left\{\lambda>0 : \frac{1}{2\pi}\int_{0}^{2\pi}\Phi\left(\frac{|G(re^{i\theta})|}{\lambda}\right)d\theta \leq 1  \right\}<\infty.    
    $$
Let $\Phi$ be a  growth function. If $\Phi$ is convex or belongs to $\mathscr{L}$ then for some $\rho \in \{1; a_{\Phi}\} $, 
\begin{equation}\label{eq:suqiaqsq8n}
H^{\Phi}(\mathbb{D})  \subseteq H^{\rho}(\mathbb{D}). 
\end{equation}

The proof of the following result is identical to that of \cite[Theorem 3.11]{djesehaqbAZ}. Therefore, the proof will be omitted.

\begin{theorem}\label{pro:main5Qpm5}
  Let $\Phi$ be a growth function such that  $\Phi(t)>0$ for all  $t>0$. If $\Phi$ is convex or belongs to $\mathscr{L}$, then for $F\in H^{\Phi}(\mathbb{C_{+}})$, the function $G$ defined by
  $$   G(\omega)= F\left(i\frac{1-\omega}{1+\omega} \right), ~~\forall~\omega \in \mathbb{D},    $$
  is in $H^{\Phi}(\mathbb{D})$. Moreover,
  $$  \|G\|_{H^{\Phi}(\mathbb{D})} \leq  \|F\|_{H^{\Phi}(\mathbb{C_{+}})}^{lux}.     $$
  \end{theorem}

Denote by $B$  the function B\^eta defined by
 $$  B(m,n) =\int_{0}^{\infty}\dfrac{u^{m-1}}{(1+u)^{m+n}}du , ~~\forall~ m, n > 0.   $$
 
 The following results can be found for example in \cite{bansahsehba}.
  
 \begin{lemme}\label{pro:main26}
  Let  $y > 0$ and $\alpha \in \mathbb{R}$. The integral 
 $$ \textit{J}_{\alpha}(y) =\int_{\mathbb{R}}\dfrac{dx}{|x+iy|^{\alpha}},    $$ 
 converges if and only if $\alpha > 1$. In this case,
 $$ \textit{J}_{\alpha}(y)=B\left(\frac{1}{2}, \frac{\alpha-1}{2} \right)y^{1-\alpha}.    $$
\end{lemme}

 \begin{lemme}\label{pro:main27}
 Let $\alpha, \beta \in \mathbb{R}$ and $t>0$. The integral 	
 \begin{equation}\label{eq:alfdyafintgama}
 \textit{I}(t) =\int_{0}^{\infty}\dfrac{y^{\alpha}}{(t+y)^{\beta}}dy, \end{equation}		
 converges if and only if $\alpha > -1$ and $\beta>\alpha+1$. In this case,
 \begin{equation}\label{eq:alfdybeitete}
 \textit{I}(t)=B(1+\alpha, \beta-\alpha-1)t^{-\beta+\alpha+1}.\end{equation}	
 \end{lemme}

Nevanlinna's class on $\mathbb{C}_{+}$, $\mathscr{N}(\mathbb{C}_{+})$ is the set of holomorphic functions $F$ on $\mathbb{ C}_{+}$ such that
$$  \sup_{y> 0}\int_{\mathbb{R}}\log\left( 1+|F(x+iy)|  \right)dx< \infty.
   $$
For $0\not\equiv F\in \mathscr{N}(\mathbb{C}_{+})$,  there exists a unique function $f$ measurable on $\mathbb{R}$ such that $\log |f| \in L^{1}\left(\mathbb{R},\frac{dt}{1+t^{2}}\right)$ and
$$ \lim_{y \to 0}F(x+iy)=f(x),    $$
     for almost all $x\in \mathbb{R}$, (see \cite{Mochizuki}).

  \begin{proposition}\label{pro:main0aplaq0}
 Let  $\Phi \in \mathscr{C}^{1}(\mathbb{R}_{+})$ be a growth function such that   $0< a_\Phi\leq b_\Phi <\infty$. The following assertions are satisfied. 
 \begin{itemize}
 \item[(i)] If  $0< a_\Phi\leq b_\Phi \leq 1$, then $H^{\Phi}(\mathbb{C}_{+}) \subset \mathscr{N}(\mathbb{C}_{+})$. 
 \item[(ii)] If  $1< a_\Phi\leq b_\Phi <\infty$, then $H^{\Phi}(\mathbb{C}_{+}) \not\subset \mathscr{N}(\mathbb{C}_{+})$. 
 \end{itemize}
 \end{proposition}
 
 \begin{proof}
 $(i)$ For  $0\not\equiv F\in H^{\Phi}(\mathbb{C}_{+})$, put  $$ F_{1}=F\chi_{0<\{|F|\leq 1\}}  \hspace*{0.25cm}\textrm{and} \hspace*{0.25cm} F_{2}=F\chi_{\{|F|\geq 1\}}.  $$
For  $z \in \mathbb{C}_{+}$, we have
$$   \log(1+|F_{1}(z)|) \leq |F_{1}(z)| \leq |F_{1}(z)|^{b_\Phi} \leq \frac{1}{\Phi(1)}\times \Phi(|F_{1}(z)|)   $$
and
  $$   \log(1+|F_{2}(z)|)= \frac{1}{a_\Phi}\log(1+|F_{2}(z)|)^{a_\Phi}\leq \frac{2^{a_\Phi}}{a_\Phi} |F_{2}(z)|^{a_\Phi} \leq \frac{2^{a_\Phi}}{a_\Phi}\frac{1}{\Phi(1)}\times \Phi(|F_{2}(z)|),   $$ 
since the function $t\mapsto \frac{\Phi(t)}{t^{a_\Phi}}$ (resp. $t\mapsto \frac{\Phi(t)}{t^{b_\Phi} }$)
  is non-decreasing (resp. non-increasing) on $\mathbb{R}_{+}^{*}$. Using the sub-additivity of the logarithmic function on $(1,\infty)$, we deduce that 
 $$ \log(1+|F(z)|)  \lesssim \log(1+|F_{1}(z)|+|F_{2}(z)|) \lesssim  \left(  \Phi(|F_{1}(z)|)+  \Phi(|F_{2}(z)|)  \right).    $$  
%since 
%$$ (1+|F_{1}(z)|) (1+|F_{2}(z)|)- (1+|F_{1}(z)|+|F_{2}(z)|) %=|F_{1}(z)||F_{2}(z)| \geq 0.      $$
It follows that $F\in  \mathscr{N}(\mathbb{C}_{+})$. Indeed,  for $y > 0$, we have
\begin{eqnarray*} \int_{\mathbb{R}}\log(1+|F(x+iy)|)dx &\lesssim&  \int_{\mathbb{R}}\Phi(|F_{1}(x+iy)|)dx + \int_{\mathbb{R}}\Phi(|F_{2}(x+iy)|)dx\\ &\lesssim& \sup_{y> 0}\int_{\mathbb{R}}\Phi(|F(x+iy)|)dx <\infty. 
\end{eqnarray*}
\medskip

$(ii)$ Let $\alpha \in \mathbb{R}$ such that $1/a_\Phi <\alpha <1$. For $z \in \mathbb{C}_{+}$, put $$   F_{\alpha}(z) =\frac{1}{(z+i)^{\alpha}}.       $$
 By construction, $F_{\alpha}$ is an analytic function on $\mathbb{C}_{+}$ and 
$$  |F_{\alpha}(z)| = \frac{1}{|x+i(1+y)|^{\alpha}}  < 1, ~~\forall~z=x+iy \in \mathbb{C}_{+}.     $$
We deduce that 
 $$ \log\left( 1+ |F_{\alpha}(z)| \right)  \geq \frac{1}{2} \frac{1}{|x+i(1+y)|^{\alpha}}, ~~\forall~z=x+iy \in \mathbb{C}_{+}     $$
and
$$  \Phi\left( |F_{\alpha}(z)|  \right) \leq \Phi(1) \frac{1}{|x+i(1+y)|^{\alpha a_\Phi}},   ~~\forall~z=x+iy \in \mathbb{C}_{+},      $$
since $|F_{\alpha}|<1$ and the function $t\mapsto \frac{\Phi(t)}{t^{a_\Phi}}$ is non-decreasing on $\mathbb{R}_{+}^{*}$. It follows that  $F_{\alpha}\in H^{\Phi}(\mathbb{C}_{+})$ and  $F_{\alpha}\not\in \mathscr{N}(\mathbb{C}_{+})$. Indeed, for  $y > 0$, we have 
$$ \int_{\mathbb{R}}\Phi\left( |F_{\alpha}(x+iy)|  \right)dx \lesssim   B\left(\frac{1}{2}, \frac{\alpha a_\Phi-1}{2} \right)(1+y)^{1-\alpha a_\Phi} \leq B\left(\frac{1}{2}, \frac{\alpha a_\Phi-1}{2} \right)< +\infty  $$
and 
$$ \int_{\mathbb{R}} \log\left( 1+ |F_{\alpha}(x+iy)| \right) dx \geq \frac{1}{2}\int_{\mathbb{R}}\frac{dx}{|x+i(1+y)|^{\alpha}} =+\infty,    $$
according to Lemma \ref{pro:main26}.
 \end{proof}

Let $f$ be a measurable function on $\mathbb{R}$. The Poisson integral  $U_{f}$ of $f$ is the function defined  by
$$  U_{f}(x+iy):=  \frac{1}{\pi}\int_{\mathbb{R}}\frac{y}{(x-t)^{2}+y^{2}}f(t)dt,~~ \forall~ x+iy \in \mathbb{C_{+}},    $$
when it makes sense.\\
If $f\in L^{1}\left(\mathbb{R},\frac{dt}{1+t^{2}}\right)$ then $U_{f}$ is a harmonic function on $\mathbb{C_{+ }}$ and
  $$  \lim_{y \to 0}U_{f}(x+iy)=f(x),       $$
  for almost all $x\in \mathbb{R}$ (see \cite{javadmas}).

\begin{lemme}[Lemma 4.1, \cite{djesehaqbAZ}]\label{pro:mainfaq5}
Let $\Phi$ be a convex growth function such that $\Phi(t)>0$ for all  $t>0$ and  $0\not\equiv F$ an analytic function on $\mathbb{C_{+}}$. The following assertions are equivalent.
\begin{itemize}
 \item[(i)] $ F\in H^{\Phi}(\mathbb{C_{+}})$.
 \item[(ii)] There exists a unique function 
    $f\in  L^{\Phi}\left(\mathbb{R}\right)$ such that   $\log|f| \in L^{1}\left(\mathbb{R},\frac{dt}{1+t^{2}}\right)$ and 
$$ F(x+iy)= U_{f}(x+iy), ~~\forall~x+iy \in \mathbb{C}_{+}.$$
\end{itemize}
Moreover, 
$$  \|F\|_{H^{\Phi}}^{lux}=\lim_{y \to 0}\|F(.+iy)\|_{L^{\Phi}}^{lux}=\|f\|_{L^{\Phi}}^{lux}.
     $$
\end{lemme}

\begin{theorem}\label{pro:mainfaaqaqaqq5}
Let $\Phi$ be a growth function such that  $\Phi(t)>0$ for all  $t>0$. If $\Phi$ is convex or belongs to $\mathscr{L}$, then for  $ 0\not\equiv F\in H^{\Phi}(\mathbb{C_{+}})$,  there exists a unique function 
     $f\in  L^{\Phi}\left(\mathbb{R}\right)$ such that   $\log|f| \in L^{1}\left(\frac{dt}{1+t^{2}}\right)$,  
  $$  f(x)=\lim_{y\to 0}F(x+iy),   $$
 for almost all $x\in \mathbb{R}$,  $f(t)\not=0$ for almost all $t\in \mathbb{R}$, 
$$  \log|F(x+iy)| \leq   \frac{1}{\pi}\int_{\mathbb{R}}\frac{y}{(x-t)^{2}+y^{2}}\log|f(t)|dt, ~~\forall~x+iy \in \mathbb{C}_{+}
    $$
and 
\begin{equation}\label{eq:sqaaqq8n}
\|F\|_{H^{\Phi}}^{lux}=\lim_{y \to 0}\|F(.+iy)\|_{L^{\Phi}}^{lux}=\|f\|_{L^{\Phi}}^{lux}.
\end{equation}
\end{theorem}

\begin{proof}
Let $ 0\not\equiv F\in H^{\Phi}(\mathbb{C_{+}})$.  There exists a unique measurable function $f$ on $\mathbb{R}$ such that $\log|f| \in L^{1}\left(\frac{dt}{1+t^{2}}\right)$ and
  $$\lim_{y \to 0}F(x+iy)=f(x),$$
 for almost all $x\in \mathbb{R}$, according to point $(i)$ of Proposition \ref{pro:main0aplaq0} and Lemma \ref{pro:mainfaq5}. 
 Suppose that there exists $A$ a measurable subset of $\mathbb{R}$ with Lebesgue measure $|A| >0$, and $$ f(x)=0, ~~\forall~x\in A. $$
 We have
 $$ +\infty=\int_{A} |\log|f(t)|| \frac{dt}{1+t^{2}}\leq \int_{\mathbb{R}} |\log|f(t)|| \frac{dt}{1+t^{2}}.    $$
 We deduce that  $\log|f| \not\in L^{1}\left(\frac{dt}{1+t^{2}}\right)$. Which is absurd. Hence, $f(t)\not=0$, for almost all $t\in \mathbb{R}$.
 For  $\omega \in \mathbb{D}$,  put $$ G(\omega)= F\left(i\frac{1-\omega}{1+\omega} \right).       $$
 Since $G \in H^{\Phi}(\mathbb{D}) \subset H^{p}(\mathbb{D})$, with  $p>0$, there exists a unique function $g\in L^{\Phi}(\mathbb{T})$ such that $\log|g| \in L^{1}(\mathbb{T})$ and
  $$   \lim_{r\to 1}G(re^{i\theta}) =g(e^{i\theta}),      $$
  for almost all $\theta\in \mathbb{R}$ and
  $$ \log|G(re^{i\theta})|\leq  \frac{1}{2\pi}\int_{-\pi}^{\pi}\frac{1-r^{2}}{1-2r\cos(u-\theta)+r^{2}}\log|g(e^{iu})|du, ~~\forall~re^{i\theta} \in \mathbb{D}.     $$
Moreover,
 \begin{equation}\label{eq:suiaaqaq8n}
 \log|g(e^{i\theta})|=   \lim_{r\to 1} \left(   \frac{1}{2\pi}\int_{-\pi}^{\pi}\frac{1-r^{2}}{1-2r\cos(u-\theta)+r^{2}}\log|g(e^{iu})|du  \right),
 \end{equation}
 for almost all $\theta\in \mathbb{R}$.\\ 
Consider  $\varphi$, the map defined by
   $$   \varphi(\omega)= i\frac{1-\omega}{1+\omega},~~\forall~\omega \in \mathbb{D} \cup \mathbb{T}\backslash\{-1\}, 
       $$
 where $\mathbb{T}$ is the complex unit circle. Note that the restriction of $\varphi$ to $\mathbb{D}$ (resp. $\mathbb{T}\backslash\{-1\}$) is an analytic function on $\mathbb{D}$ with values in $\mathbb {C}_{+}$ (resp. a homeomorphism from $\mathbb{T}\backslash\{-1\}$ onto $\mathbb{R}$).\\
For  $z=x+iy\in \mathbb{C}_{+}$ and  $\omega=re^{iu} \in \mathbb{D}$ such that  $ z=i\frac{1-\omega}{1+\omega}$, using 
  $$ y=\frac{1-r^{2}}{1+r^{2}+2r\cos u}         $$
 and the Relation (\ref{eq:suiaaqaq8n}),  we deduce  that
 $$  |f(x)|=|g\circ\varphi^{-1}(x)|,   $$
 for almost all $x\in \mathbb{R}$. Therefore,
 \begin{equation}\label{eq:suiaqaaqaq8n}
 \log|F(x+iy)| \leq  \frac{1}{\pi}\int_{\mathbb{R}}\frac{y}{(x-t)^{2}+y^{2}}\log|f(t)|dt, ~~\forall~x+iy\in \mathbb{C}_{+}.
 \end{equation}
 Indeed
 \begin{align*}
  \log|F(x+iy)|&= \log|G(re^{iu})|     \\
  &\leq  \frac{1}{2\pi}\int_{-\pi}^{\pi}\frac{1-r^{2}}{1-2r\cos(u-\theta)+r^{2}}\log|g(e^{i\theta})|d\theta \\
  &=\frac{1}{\pi}\int_{\mathbb{R}}\frac{y}{(x-t)^{2}+y^{2}}\log|g\circ\varphi^{-1}(t)|dt
  = \frac{1}{\pi}\int_{\mathbb{R}}\frac{y}{(x-t)^{2}+y^{2}}\log|f(t)|dt.
  \end{align*}
Let us prove Relation (\ref{eq:sqaaqq8n}). 
By Fatou's lemma,  we have $$ \int_{\mathbb{R}}\Phi\left( \frac{|f(x)|}{\|F\|_{H^{\Phi}}^{lux}}  \right)dx \leq \liminf_{y \to 0}\int_{\mathbb{R}}\Phi\left( \frac{|F(x+iy)|}{\|F\|_{H^{\Phi}}^{lux}}  \right)dx \leq   \sup_{y>0}\int_{\mathbb{R}}\Phi\left( \frac{|F(x+iy)|}{\|F\|_{H^{\Phi}}^{lux}}\right)dx \leq 1. $$
We deduce that $ f\in L^{\Phi}(\mathbb{R})$ and 
\begin{equation}\label{eq:suiaaq8n}
\|f\|_{L^{\Phi}}^{lux} \leq \|F\|_{H^{\Phi}}^{lux}.
\end{equation}
Put 
 $$  \Phi_{\rho}(t)=\Phi\left(t^{1/\rho}\right), ~~\forall~t\geq 0,      $$
 where $\rho=1$ if $\Phi$ is convex and $\rho=a_{\Phi}$ if  $\Phi \in \mathscr{L}$.\\
From Jensen's inequality and also from the Relation (\ref{eq:suiaqaaqaq8n}), we deduce that
$$  |F(x+iy)|^{\rho} \leq   \frac{1}{\pi}\int_{\mathbb{R}}\frac{y}{(x-t)^{2}+y^{2}}|f(t)|^{\rho}dt, ~~\forall~x+iy \in \mathbb{C}_{+}.        $$
Fix $y>0$. We have 
\begin{align*}
\int_{\mathbb{R}}\Phi\left(\frac{|F(x+iy)|}{\|f\|_{L^{\Phi}}^{lux}}\right)dx &\leq \int_{\mathbb{R}}\Phi_{\rho}\left(\frac{1}{\pi}\int_{\mathbb{R}}\frac{y}{(x-t)^{2}+y^{2}}\left(\frac{|f(t)|}{\|f\|_{L^{\Phi}}^{lux}}\right)^{\rho}dt\right)dx\\
&\leq \int_{\mathbb{R}}\frac{1}{\pi}\int_{\mathbb{R}}\frac{y}{(x-t)^{2}+y^{2}}\Phi_{\rho}\left(\left(\frac{|f(t)|}{\|f\|_{L^{\Phi}}^{lux}}\right)^{\rho}\right)dtdx
\\
&= \int_{\mathbb{R}}\Phi\left(\frac{|f(t)|}{\|f\|_{L^{\Phi}}^{lux}}\right)\left(\frac{1}{\pi}\int_{\mathbb{R}}\frac{y}{(x-t)^{2}+y^{2}}dx\right)dt\\
&=
\int_{\mathbb{R}}\Phi\left(\frac{|f(t)|}{\|f\|_{L^{\Phi}}^{lux}}\right)dt \leq 1.
\end{align*}
We deduce that
\begin{equation}\label{eq:eaqaaqaqwl}
   \|F\|_{H^{\Phi}}^{lux} \leq \|f\|_{L^{\Phi}}^{lux}. \end{equation} 
From Relations (\ref{eq:suiaaq8n}) and (\ref{eq:eaqaaqaqwl}) and also from Proposition \ref{pro:main6apmqaaqq4}, it follows that  
 $$ \|F\|_{H^{\Phi}}^{lux}=\lim_{y \to 0}\|F(.+iy)\|_{L^{\Phi}}^{lux}=\|f\|_{L^{\Phi}}^{lux}.
    $$
\end{proof}
 
\begin{lemme}\label{pro:main118}
Let $\alpha>-1$ and  $\Phi$ a one-to-one growth function. If $\Phi$ is convex or belongs to $\mathscr{L}$, then there exists a constant $C:=C_{\alpha, \Phi}>1$ such that for 
$ F \in A^{\Phi}_{\alpha}(\mathbb{C_{+}})$, 
 \begin{equation}\label{eq:inegalitedehay}
 |F(x+iy)|\leq C\Phi^{-1}\left(\frac{1}{ y^{2+\alpha}}\right)\|F\|_{A^{\Phi}_{\alpha}}^{lux},~~\forall~x+iy \in \mathbb{C_{+}}.
 \end{equation}
 \end{lemme}
 
 \begin{proof}
For $t\geq 0$, put
 $$  \Phi_{\rho}(t)=\Phi\left(t^{1/\rho}\right),      $$
 where $\rho=1$ if $\Phi$ is convex and $\rho=a_{\Phi}$ if  $\Phi \in \mathscr{L}$.\\
Let $0\not\equiv F \in A^{\Phi}_{\alpha}(\mathbb{C_{+}})$. Fix $z_{0}=x_{0}+iy_{0}\in \mathbb{C_{+}}$ and put $r=\dfrac{y_{0}}{2}$. Since $|F|^{\rho}$ is subharmonic on $\mathbb{C}_ {+}$, we have
 $$ |F(z_{0})|^{\rho} \leq \frac{1}{\pi r^{2}}\int \int_{\overline{\mathcal{D}(z_{0}, r)}}|F(u+iv)|^{\rho}dudv.   $$
For $u+iv \in \overline{\mathcal{D}(z_{0}, r)}$, we have
 $$  r \leq v \leq 3r  \Rightarrow 0<\frac{1}{v^{\alpha}}\leq 2^{\alpha}\times\frac{1}{y_{0}^{\alpha}},~~ \text{if}~~ \alpha  \geq 0 \hspace*{0.5cm}\textrm{and} \hspace*{0.5cm} 0<\frac{1}{v^{\alpha}}\leq  \left(\frac{2}{3}\right)^{\alpha}\times\frac{1}{y_{0}^{\alpha}},~~ \text{if} ~~ -1<\alpha < 0.    $$
We deduce that 
\begin{equation}\label{eq:espqsa2ces1}
 0<\dfrac{1}{v^{\alpha}}  \leq C_{\alpha}\dfrac{1}{y_{0}^{\alpha}},~~\forall~u+iv \in \overline{\mathcal{D}(z_{0}, r)},\end{equation}
 where $  C_{\alpha}:= \max\left\{ 2^{\alpha} ;  \left(2/3\right)^{\alpha}\right\}$.
By Jensen's inequality, we have 
\begin{align*}
 \Phi\left(\left(\frac{\pi}{4C_{\alpha}}\right)^{1/\rho} \times\frac{|F(z_{0})|}{\|F\|_{A_{\alpha}^{\Phi}}^{lux}}\right)   &\leq \frac{\pi}{4C_{\alpha}}\Phi_{\rho}\left( \frac{1}{\pi r^{2}}\int \int_{\overline{\mathcal{D}(z_{0}, r)}}\left(\frac{|F(u+iv)|}{\|F\|_{A_{\alpha}^{\Phi}}^{lux}}\right)^{\rho}dudv \right) \\
 &\leq \frac{\pi}{4C_{\alpha}}\times\frac{4}{\pi y_{0}^{2}}\times \dfrac{C_{\alpha}}{y_{0}^{\alpha}} \int \int_{\overline{\mathcal{D}(z_{0}, r)}}\Phi\left(\frac{|F(u+iv)|}{\|F\|_{A_{\alpha}^{\Phi}}^{lux}}\right)v^{\alpha}dudv \\
  &\leq  \frac{1}{y_{0}^{2+\alpha}} \int_{\mathbb{C_{+}}}\Phi\left(\frac{|F(u+iv)|}{\|F\|_{A_{\alpha}^{\Phi}}^{lux}}\right)dV_{\alpha}(u+iv) \leq \frac{1}{y_{0}^{2+\alpha}}.
  \end{align*}
We deduce that $$ |F(z_{0})|\leq \left( \frac{4C_{\alpha}}{\pi} \right)^{1/\rho}\Phi^{-1}\left(\frac{1}{ y_{0}^{2+\alpha}}\right)\|F\|_{A^{\Phi}_{\alpha}}^{lux}.           $$
 \end{proof}

\begin{proposition}\label{pro:main1aq18}
Let $\alpha>-1$. There exist $C:=C_{\alpha}>0$ and $\beta \in \left\{0, 1/3\right\}$  such that for any analytic function $F$ on $\mathbb{C_{ +}}$ and for all $0<\gamma<\infty$,
\begin{equation}\label{eq:fonciaqnpmal}
|F(z)|^{\gamma} \leq C\mathcal{M}_{V_{\alpha}}^{\mathcal{D}^{\beta}}\left(|F|^{\gamma}\right)(z), ~~\forall~z \in \mathbb{C_{+}}.
\end{equation}
\end{proposition}

\begin{proof}
Let   $0<\gamma<\infty$ and  $0\not\equiv F$ an analytic function on $\mathbb{C_{+}}$. Fix $z_{0}=x_{0}+iy_{0}\in \mathbb{C_{+}}$ and $r=\dfrac{y_{0}}{2}$. From Relation (\ref{eq:espqsa2ces1}) we have
 $$ 0<\dfrac{1}{v^{\alpha}}  \leq \max\left\{ 2^{\alpha} ;  \left(2/3\right)^{\alpha}\right\}\dfrac{1}{y_{0}^{\alpha}},~~\forall~u+iv \in \overline{\mathcal{D}(z_{0}, r)}.       $$
Let $I$ be an interval centered at $x_{0}$ and of length
$|I|=2y_{0}$. Consider $Q_{I}$ the Carleson square associated with $I$. According to Lemma \ref{pro:main80}, there exist $\beta \in \left\{0, 1/3\right\}$ and $J \in \mathcal{D}^{\beta}$ such that $I\subset J$ and $|J|\leq 6|I|$. From  Relation (\ref{eq:fo4n1l}) we have
$$ |Q_{J}|_{\alpha}=\frac{1}{1+\alpha}|J|^{2+\alpha} \leq \frac{6^{2+\alpha}}{1+\alpha}|I|^{2+\alpha}
= \frac{12^{2+\alpha}}{1+\alpha}y_{0}^{2+\alpha}.       $$
Since  $|F|^{\gamma}$ is subharmonic on $\mathbb{C}_ {+}$ and  $\overline{\mathcal{D}(z_{0}, r)}$ is contained in $Q_{I}$  we have 
\begin{align*}
|F(z_{0})|^{\gamma} &\leq \frac{1}{\pi r^{2}}\int \int_{\overline{\mathcal{D}(z_{0}, r)}}|F(u+iv)|^{\gamma}dudv \\
&\leq \frac{4}{\pi y_{0}^{2}}\times \dfrac{\max\left\{ 2^{\alpha} ;  \left(2/3\right)^{\alpha}\right\}}{y_{0}^{\alpha}} \int \int_{\overline{\mathcal{D}(z_{0}, r)}}|F(u+iv)|^{\gamma}v^{\alpha}dudv \\
&\leq  C_{\alpha}\frac{\chi_{Q_{J}}(z_{0})}{|Q_{J}|_{\alpha}}\int \int_{Q_{J}}|F(u+iv)|^{\gamma}v^{\alpha}dudv \leq  C_{\alpha} \mathcal{M}_{V_{\alpha}}^{\mathcal{D}^{\beta}}\left(|F|^{\gamma}\right)(z_{0}), 
 \end{align*}
where $C_{\alpha}:=\frac{4}{\pi}\times\frac{12^{2+\alpha}}{1+\alpha}\times \max\left\{ 2^{\alpha} ;  \left(2/3\right)^{\alpha}\right\}$.
\end{proof}

\begin{proposition}\label{pro:main10aqa9}
Let  $\alpha>-1$ and $\Phi$ a one-to-one growth function. If $\Phi$ is convex or belongs to $\mathscr{L}$  then there exists some constants  $\rho \in \{1; a_{\Phi}\} $ and  
\begin{equation}\label{eq:fo4naqm1l}
C_{\alpha}:= B\left(1+\alpha, 2+\alpha\right)B\left(\frac{1}{2}, \frac{3+2\alpha}{2}\right), \end{equation}
such that for all  $z=x+iy\in \mathbb{C}_{+}$ the functions $F_{z}$ and $G_{z}$  defined respectively by 
\begin{equation}\label{eq:fo4naaqqm1l}
F_{z}(\omega)=\Phi^{-1}\left(\frac{1}{\pi y}\right) \frac{y^{2/\rho}}{ (\omega-\overline{z})^{2/\rho}} ,~~ \forall~ \omega\in \mathbb{C_{+}} \end{equation}
and
\begin{equation}\label{eq:fo4qm1l}
G_{z}(\omega)=\Phi^{-1}\left(\frac{1}{C_{\alpha}y^{2+\alpha}}\right) \frac{y^{(4+2\alpha)/\rho}}{ (\omega-\overline{z})^{(4+2\alpha)/\rho}},~ \forall~ \omega\in \mathbb{C_{+}},
 \end{equation}
 are  analytic functions belong respectively to  $H^{\Phi}(\mathbb{C_{+}})$ and
  $A^{\Phi}_{\alpha}(\mathbb{C_{+}})$. Moreover,   $ \|F_{z}\|_{H^{\Phi}}^{lux}\leq 1$ and
 $ \|G_{z}\|_{A^{\Phi}_{\alpha}}^{lux}\leq 1$.
 \end{proposition}

\begin{proof}
Fix  $z=x+iy\in \mathbb{C}_{+}$. By construction $F_{z}$  ad  $G_{z}$ are analytic functions which does not vanish on $\mathbb{C_{+}}$. For   $\omega=u+iv \in \mathbb{C_{+}}$, we have $$  \frac{y^{2}}{|(u-x)+i(y+v)|^{2}} \leq 1.   $$
Put  $\rho=1$ if $\Phi$ is convex and $\rho=a_{\Phi}$ if  $\Phi \in \mathscr{L}$, and 
 $$ C_{\alpha}:= B\left(1+\alpha, 2+\alpha\right)B\left(\frac{1}{2}, \frac{3+2\alpha}{2}\right).
       $$
Since the function $t\mapsto \frac{\Phi(t)}{t^{\rho}}$
 is non-decreasing on $\mathbb{R}_{+}^{*}$, we deduce that
$$  \int_{\mathbb{R}}\Phi\left(|F_{z}(u+iv)|\right)du \lesssim   \frac{y}{\pi }\int_{\mathbb{R}}\frac{1}{|(u-x)+i(y+v)|^{2}}du  $$
and  
$$  \int_{\mathbb{C}_{+}}\Phi(|G_{z}(\omega)|)dV_{\alpha}(\omega)  \lesssim  \frac{y^{2+\alpha}}{C_{\alpha}}\int_{0}^{\infty}\left( \int_{\mathbb{R}}\frac{du}{ |(u-x)+i(v+y)|^{4+2\alpha}}\right)v^{\alpha}dv.  $$
According to Lemma \ref{pro:main26}, we have  
$$   \int_{\mathbb{R}}\frac{1}{|(u-x)+i(y+v)|^{2}}du
= B\left(\frac{1}{2}, \frac{1}{2} \right)\frac{1}{y+v}      $$
and 
$$ \int_{\mathbb{R}}\dfrac{du}{|(u-x)+i(v+y)|^{4+2\alpha}}= B\left(\frac{1}{2}, \frac{3+2\alpha}{2}\right)\frac{1}{(v+y)^{3+2\alpha}}.        $$
We deduce that 
$$  \int_{\mathbb{R}}\Phi\left(|F_{z}(u+iv)|\right)du \lesssim  1,  ~~\forall~ v > 0 $$
and
$$  \int_{\mathbb{C}_{+}}\Phi(|G_{z}(\omega)|)dV_{\alpha}(\omega)  \lesssim 1,   $$
since 
 $$ \int_{0}^{\infty}\dfrac{v^{\alpha}}{(y+v)^{3+2\alpha}}dv = B(1+\alpha, 2+\alpha)\frac{1}{y^{2+\alpha}}, $$
thanks to Lemma \ref{pro:main27}. Therefore,  $F_{z} \in H^{\Phi}(\mathbb{C_{+}})$ with $\|F_{z}\|_{H^{\Phi}}^{lux}\leq 1$ and  $G_{z}\in A^{\Phi}_{\alpha}(\mathbb{C_{+}})$ with $\|G_{z}\|_{A^{\Phi}_{\alpha}}^{lux}\leq 1$.
\end{proof}
 
 \section{Some characterizations of Carleson measures.}
In this section, we give among others, a general characterization of an ($s,\Phi$)-Carleson measure.
\begin{proposition}\label{pro:main14a9}
Let $s>0$, $\alpha > -1$ and $\Phi_{1},\Phi_{2}$ be two one-to-one growth functions. The following assertions are equivalent.
\begin{itemize}
\item[(i)] $V_{\alpha}$ is a $(s,\Phi_{2} \circ \Phi_{1}^{-1})-$Carleson measure.
\item[(ii)] There exists a constant $C>0$ such that for all $t > 0$ 
\begin{equation}\label{eq:ibejection1q22}
  \Phi_{1}^{-1}(t^{s}) \leq \Phi_{2}^{-1}(Ct^{2+\alpha})
.\end{equation}
\end{itemize}
\end{proposition}

\begin{proof} 
Show that (i) implies (ii). \\
Fix $t > 0$ and let  $I$ an interval such that  $|I|=\frac{1}{t}$. Consider $Q_{I}$ the Carleson square associated with $I$. 
Since $V_{\alpha}$ is a  $(s,\Phi_{2} \circ \Phi_{1}^{-1})-$Carleson measure, we have
$$  V_{\alpha}(Q_{I}) \leq \dfrac{C}{\Phi_{2} \circ \Phi_{1}^{-1}\left(\frac{1}{|I|^{s}}\right)} 
\Rightarrow  \dfrac{1}{1+\alpha}\dfrac{1}{t^{2+\alpha}} \leq \dfrac{C}{\Phi_{2} \circ \Phi_{1}^{-1}(t^{s})}  
\Rightarrow \Phi_{2} \circ \Phi_{1}^{-1}(t^{s}) \leq (1+\alpha) C t^{\alpha+2}.
          $$
For the converse,  we suppose that (ii) is true and prove (i).\\
Let $I$ be an interval of nonzero length and $Q_{I}$ the Carleson square associated with $I$.
Since the inequality (\ref{eq:ibejection1q22}) is satisfied, we have          
\begin{align*}
\Phi_{1}^{-1}\left( \frac{1}{|I|^{s}} \right) \leq \Phi_{2}^{-1}\left(\frac{C}{|I|^{\alpha+2}} \right)  	&\Rightarrow  \Phi_{2} \circ \Phi_{1}^{-1}\left(\frac{1}{|I|^{s}}\right) \leq \frac{C}{|I|^{\alpha+2}} \\
&\Rightarrow  \Phi_{2} \circ \Phi_{1}^{-1}\left(\frac{1}{|I|^{s}}\right) \leq \frac{C}{(1+\alpha)V_{\alpha}(Q_{I})} \\	
&\Rightarrow V_{\alpha}(Q_{I}) \leq \frac{C'}{\Phi_{2} \circ \Phi_{1}^{-1}\left(\frac{1}{|I|^{s}}\right)}.	
\end{align*}
\end{proof}

\begin{proposition}\label{pro:main49qa2}
Let $s\geq 1$ and  $\Phi\in \mathscr{U}$. Put
$$  d\mu(x+iy)=\dfrac{dx dy}{y^{2} \Phi\left(\frac{1}{y^{s}}\right)},~~  \forall~x+iy\in \mathbb{C}_{+}.   $$
If  $\Phi\in \nabla_2$ then $\mu$ is a
 measure $(s, \Phi)-$Carleson. In particular, the converse is true for $s=1$.
\end{proposition}

\begin{proof}
Put
$$ \widetilde{\Omega}(t)=\frac{1}{ \Phi\left(\frac{1}{t}\right)}, ~\forall~t>0 \hspace*{0.5cm}\textrm{and} \hspace*{0.5cm}  \widetilde{\Omega}(0)=0.      $$
According to Proposition \ref{pro:main44},  $\Phi\in \mathscr{U} \cap \nabla_2$.\\
Let $I$ be an interval of nonzero length and $Q_{I}$ the Carleson square associated with $I$.  We have
\begin{align*}
\mu(Q_{I}) &=  \int_{0}^{|I|}\int_{I}\frac{\widetilde{\Omega}(y^{s})}{y^{2}}dx dy = |I|\int_{0}^{|I|}\frac{\widetilde{\Omega}(y^{s})}{y^{2s}}y^{s-1}y^{s-1}dy \\
&\leq  s^{-1}|I|^{s} \int_{0}^{|I|^{s}}\frac{\widetilde{\Omega}(y)}{y^{2}}dy 
\leq s^{-1}|I|^{s}C\frac{\widetilde{\Omega}(|I|^{s})}{|I|^{s}}=\frac{C/s}{\Phi\left(\frac{1}{|I|^{s}}\right)},  
\end{align*}
thanks to Lemma \ref{pro:main40l6aqlj}.
In particular, for $s=1$, we have
$$  \mu(Q_{I}) \lesssim \widetilde{\Omega}(|I|)  \Leftrightarrow  \int_{0}^{|I|}\frac{\widetilde{\Omega}(y)}{y^{2}}dy \lesssim \frac{\widetilde{\Omega}(|I|)}{|I|}.  $$
\end{proof}

\begin{lemme}\label{pro:main132}
Let  $\alpha >-1$,  $\Phi\in \mathscr{U}$ and  $\mu$ be a positive Borel measure on $\mathbb{C}_{+}$.  Put
$$ \widetilde{\Omega}(t)=\frac{1}{ \Phi\left(\frac{1}{t}\right)}, ~\forall~t>0 \hspace*{0.5cm}\textrm{and} \hspace*{0.5cm}  \widetilde{\Omega}(0)=0.      $$
The following assertions are satisfied
\begin{itemize}
\item[(i)]  $\mu$  is a measure $\Phi-$Carleson if and only if there exists a constant $C_{1}>0$ such that for all  $f\in L^{1}\left(\mathbb{R},\frac{dt}{1+t^{2}}\right)$ and any $\lambda >0$,
\begin{equation}\label{eq:ibejection}
\mu\left( \left\{z\in \mathbb{C}_{+} :  |U_{f}(z)| > \lambda \right\} \right) \leq C_{1}\widetilde{\Omega}\left(|\{ x\in \mathbb{R} : \mathcal{M}_{HL}(f)(x) > \lambda  \}  |\right),\end{equation}
where $U_{f}$ is the Poisson integral of $f$.	
\item[(ii)]  $\mu$  is a measure $(\alpha,\Phi)-$Carleson if and only if there exists a constant $C_{2}>0$ such that for  $f \in L^{\Phi}\left(\mathbb{C_{+}}, dV_{\alpha}\right)$ and $\lambda >0$,
\begin{equation}\label{eq:ibejectialfa}
\mu\left(\left\{z\in \mathbb{C}_{+} :  \mathcal{M}_{V_{\alpha}}^{\mathcal{D}^{\beta}}(f)(z) > \lambda \right\}\right) \leq C_{2}\widetilde{\Omega}\left(\left|\left\{z\in \mathbb{C}_{+} :  \mathcal{M}_{V_{\alpha}}^{\mathcal{D}^{\beta}}(f)(z) > \lambda \right\}\right|_{\alpha}\right).\end{equation}
\end{itemize}
\end{lemme}

\begin{proof}
(i) That $\mu$  is a measure $\Phi-$Carleson implies that (\ref{eq:ibejection}) holds, has already been proved in \cite[Lemma 4.2]{djesehb}.\\
Suppose the inequality (\ref{eq:ibejection}) is satisfied and show that $\mu$  is a measure $\Phi-$Carleson.\\
Let $I$ be an interval of $\mathbb{R}$ of non-zero length and $Q_{I}$ the Carleson square associated with $I$. Put
$$ \lambda= \frac{1}{2} \Phi^{-1}\left(\frac{1}{|I|}\right)    $$
and
$$ f=2\lambda\chi_{I}.   $$
By construction  $f\in L^{\Phi}(\mathbb{R})$ and $\|f\|_{L^{\Phi}}^{lux}\leq 1$. Indeed
$$ \int_{\mathbb{R}}\Phi(|f(x)|)dx =  \int_{I}\Phi \left(\Phi^{-1}\left(\frac{1}{|I|}\right)\right)dx = 1.       $$
Let $x_{0}+iy_{0} \in Q_{I}$. We have  
$$ \lambda < f(x_{0}) =\liminf_{y \to 0}U_{f}(x_{0}+iy) \leq U_{f}(x_{0}+iy_{0}), $$
where $U_{f}$ is the Poisson integral of $f$. We deduce that 
$$  Q_{I} \subset \left\{z\in \mathbb{C}_{+} :  |U_{f}(z)| > \lambda \right\}.   $$
Since inequality (\ref{eq:ibejection}) is satisfied, we have
\begin{align*}
\mu(Q_{I}) &\lesssim \mu\left( \left\{z\in \mathbb{C}_{+} :  |U_{f}(z)| > \lambda \right\} \right) \\
&\lesssim \widetilde{\Omega}\left(\left|\left\{ x\in \mathbb{R} : \mathcal{M}_{HL}(f)(x) >  \lambda \right\}  \right|\right) \\
&\lesssim\widetilde{\Omega} \left(  \frac{1}{\Phi\left(\lambda \right)}  \right)
\lesssim\widetilde{\Omega}\left( |I|  \right).
\end{align*}
\medskip

(ii) Again, that $\mu$  is a measure $(\alpha,\Phi)-$Carleson implies that (\ref{eq:ibejectialfa}) holds was proved in \cite[Lemma 4.3]{djesehb}. Let us prove the converse. Let $I$ be an interval of nonzero length and $Q_{I}$ the Carleson square associated with $I$. Put
$$ \lambda= \frac{1}{2} \Phi^{-1}\left(\frac{1+\alpha}{|I|^{2+\alpha}}\right)    $$
and
$$ f=2\lambda\chi_{Q_{I}}.   $$
By construction  $f\in L^{\Phi}(\mathbb{C_{+}}, dV_{\alpha})$ and $\|f\|_{L_{\alpha}^{\Phi}}^{lux}\leq 1$. Indeed
$$ \int_{\mathbb{C}_{+}}\Phi(|f(z)|)dV_{\alpha}(z) \leq  \int_{Q_{I}}\Phi \left(\Phi^{-1}\left(\frac{1+\alpha}{|I|^{2+\alpha}}\right)\right)dV_{\alpha}(z) = 1.       $$
By Lemma \ref{pro:main80}, there are $\beta \in \left\{0, 1/3\right\}$ and $J \in \mathcal{D}^{\beta}$ such that $I\subset J$ and $|J|\leq 6|I|$. Consider  $Q_{J}$ the Carleson square associated with $J$.
Let $z \in Q_{I}$. We have
$$ \lambda <  \frac{\chi_{Q_{I}}(z)}{|Q_{I}|_{\alpha}}\int_{Q_{I}}f(\omega)dV_{\alpha}(\omega) \lesssim  \frac{\chi_{Q_{J}}(z)}{|Q_{J}|_{\alpha}}\int_{Q_{J}}f(\omega)dV_{\alpha}(\omega) \lesssim \mathcal{M}_{V_{\alpha}}^{\mathcal{D}^{\beta}}f(z). $$
We deduce that 
$$  Q_{I}\subset \left\{ z\in \mathbb{C}_{+} :  \mathcal{M}_{V_{\alpha}}^{\mathcal{D}^{\beta}}f(z) > \lambda \right\}.       $$
Since the inequality (\ref{eq:ibejectialfa}) is satisfied and by Chebychev's inequality, we have
\begin{align*}
\mu(Q_{I}) &\lesssim \mu\left(\left\{ z\in \mathbb{C}_{+} :  \mathcal{M}_{V_{\alpha}}^{\mathcal{D}^{\beta}}f(z) > \lambda \right\}\right)\\
&\lesssim\widetilde{\Omega}\left(\left|\left\{ z\in \mathbb{C}_{+} :  \mathcal{M}_{V_{\alpha}}^{\mathcal{D}^{\beta}}f(z) > \lambda \right\}\right|_{\alpha}\right)\\
&\lesssim\widetilde{\Omega}\left( \frac{1}{\Phi\left(\Phi^{-1}\left(\frac{1}{|I|^{2+\alpha}}\right)\right)}\right) 
\lesssim\widetilde{\Omega}\left(|I|^{2+\alpha}\right).
\end{align*}

\end{proof}
The following is a generalization of \cite[Theorem 4.1]{djesehb}
\begin{theorem}\label{pro:main14paq1}
Let $s>0$ be a real,  $\Phi_{1}, \Phi_{2}$ two one-to-one growth functions and  $\mu$ a positive Borel measure on $\mathbb{C}_{+}$. If $\Phi_{2} \in \mathscr{L} \cup\mathscr{U}$ and $\Phi_{1}$ is convex or belongs  $\mathscr{L}$  then the following assertions are equivalent.
\begin{itemize}
\item[(i)]  $\mu$ is a $(s,\Phi_{2}\circ\Phi^{-1}_{1})-$Carleson measure.
\item[(ii)] There exist some constants $\rho \in \{1; a_{\Phi_{1}}\} $ and $C:=C_{s,\Phi_{1},\Phi_{2}}>0$  such that for all $z=x+iy \in \mathbb{C_{+}}$
\begin{equation}\label{eq:ualphispleson}
 \int_{\mathbb{C}_{+}}\Phi_{2}\left(\Phi^{-1}_{1}\left(\frac{1}{y^{s}}\right) \dfrac{y^{2s/\rho}}{ |\omega-\overline{z}|^{2s/\rho}}\right) d\mu(\omega) \leq C .\end{equation}
\end{itemize}
 \end{theorem}

\proof
Show that (ii) implies (i).
We assume that the inequality (\ref{eq:ualphispleson}) holds.\\
Let $I$ be an interval of nonzero length and $Q_{I}$ its Carleson square.\\
Fix $z_{0}=x_{0}+iy_{0} \in \mathbb{C}_{+}$ and we assume that  $x_{0}$ is the center of $I$ and $|I|=2y_{0}$.\\
Let $\omega=u+iv \in Q_{I}$. We have 
$$ |\omega-\overline{z_{0}}|^{2}=|(u-x_{0})+i(v+y_{0})|^{2} \leq y_{0}^{2}+(3y_{0})^{2}=10y_{0}^{2}.  $$
It follows that 
$$ 1 \leq 10^{s/\rho} \dfrac{y^{2s/\rho}_{0}}{ |\omega-\overline{z_{0}}|^{2s/\rho}}.   $$
Since $\Phi_{1}^{-1}$ is increasing and $t\mapsto\frac{\Phi_{2}(t)}{t^{b_{\Phi_{2}}}}$ is non-increasing on $\mathbb{R}_{+}^{*}$, we have
\begin{align*}
\Phi_{2} \circ \Phi_{1}^{-1}\left(\frac{1}{|I|^{s}}\right) &\leq  \Phi_{2}\left( \Phi_{1}^{-1}\left(\frac{1}{y_{0}^{s}}\right)\dfrac{10^{s/\rho}y^{2s/\rho}_{0}}{ |\omega-\overline{z_{0}}|^{2s/\rho}}\right) \\
&\leq 10^{sb_{\Phi_{2}}/\rho} \Phi_{2}\left( \Phi_{1}^{-1}\left(\frac{1}{y_{0}^{s}}\right)\dfrac{y^{2s/\rho}_{0}}{ |\omega-\overline{z_{0}}|^{2s/\rho}}\right).
\end{align*}
We deduce that
$$  \Phi_{2} \circ \Phi_{1}^{-1}\left(\frac{1}{|I|^{s}}\right) \leq 10^{sb_{\Phi_{2}}/\rho} \Phi_{2}\left( \Phi_{1}^{-1}\left(\frac{1}{y_{0}^{s}}\right)\dfrac{y^{2s/\rho}_{0}}{ |\omega-\overline{z_{0}}|^{2s/\rho}}\right),  ~~ \forall~ \omega \in Q_{I}.          $$
Since the inequality (\ref{eq:ualphispleson}) is satisfied, we have
\begin{align*}
\Phi_{2} \circ \Phi_{1}^{-1}\left(\frac{1}{|I|^{s}}\right)\mu(Q_{I}) &= \int_{Q_{I}}\Phi_{2} \circ \Phi_{1}^{-1}\left(\frac{1}{|I|^{s}}\right)d\mu(\omega)\\ &\leq 10^{sb_{\Phi_{2}}/\rho} \int_{\mathbb{C_{+}}}\Phi_{2}\left( \Phi_{1}^{-1}\left(\frac{1}{y_{0}^{s}}\right)\dfrac{y^{2s/\rho}_{0}}{ |\omega-\overline{z_{0}}|^{2s/\rho}}\right)d\mu(\omega)
\leq 10^{sb_{\Phi_{2}}/\rho}C_{2}.
\end{align*}
We deduce that $$ \mu(Q_{I}) \leq \frac{10^{sb_{\Phi_{2}}/\rho}C_{2}}{\Phi_{2} \circ \Phi_{1}^{-1}\left(\frac{1}{|I|^{s}}\right)}.                   $$
For the converse,  we assume that the inequality (\ref{eq:ual65iaq5rsleson}) holds.\\
Put
\[ \rho =
\left\{
\begin{array}{ll} 1 & \mbox{ if $\Phi_{1}$ is convex }\\ a_{\Phi_{1}} & \mbox{ if $\Phi_{1} \in \mathscr{L}$}
\end{array}
\right. \]
Fix $z_{0}=x_{0}+iy_{0} \in \mathbb{C}_{+}$ and let $j \in \mathbb{N}$.
Consider $I_{j}$ the centered interval $x_{0}$ with $|I_{j}|=2^{j+1}y_{0}$ and $Q_{I_{j}} $ its Carleson square. Put $$ E_{j}:=Q_{I_{j}}\backslash Q_{I_{j-1}},~~\forall ~j\geq 1 ~~~\text{and}~~~ E_{0}=Q_{I_{0}}.  $$
Fix  $j \in \mathbb{N}$ and let  $\omega=u+iv \in \mathbb{C}_{+}$.\\
If $\omega \in E_{0}$ then we have  $$  |\omega-\overline{z_{0}}|^{2}=|(u-x_{0})+i(v+y_{0})|^{2}\geq (v+y_{0})^{2} \geq y_{0}^{2}\geq 2^{-2}y_{0}^{2}.  $$
If $ \omega\in E_{j}$ with $j\geq 1$ then we have  $$ |\omega-\overline{z_{0}}|^{2}=|(u-x_{0})+i(v+y_{0})|^{2}\geq (u-x_{0})^{2} \geq 2^{2(j-1)}y_{0}^{2}.   $$
We deduce that
$$ \frac{y_{0}^{2s/\rho}}{|\omega-\overline{z_{0}}|^{2s/\rho}} \leq \frac{1}{2^{2(j-1)s/\rho}}, ~~ \forall~ \omega \in E_{j},~ \forall~j \geq 0.     $$
Fix  $j \in \mathbb{N}$ and let $\omega \in E_{j}$. Since the functions $t\mapsto\frac{\Phi_{1}^{-1}(t)}{t^{1/\rho}}$ and  $t\mapsto\frac{\Phi_{2}(t)}{t^{b_{\Phi_{2}}}}$ are non-increasing on $\mathbb{R}_{+}^{*}$ and  
 $t\mapsto\frac{\Phi_{2}(t)}{t^{a_{\Phi_{2}}}}$ is  non-decreasing on  $\mathbb{R}_{+}^{*}$, we have
\begin{align*}
\Phi_{2}\left(  \Phi^{-1}_{1}\left(\dfrac{1}{y^{s}_{0}}\right) \frac{y_{0}^{2s/\rho}}{|\omega-\overline{z_{0}}|^{2s/\rho}}\right) &\leq \Phi_{2}\left(  \Phi^{-1}_{1}\left(\dfrac{1}{y^{s}_{0}}\right) \frac{1}{2^{2(j-1)s/\rho}}\right)\\
&=\Phi_{2}\left(\Phi^{-1}_{1}\left(\dfrac{1}{y^{s}_{0}}\right)\frac{1}{2^{(j+1)s/\rho}}\times \frac{1}{2^{js/\rho}}\times \frac{1}{2^{-3s/\rho}}\right)\\
&\leq \frac{1}{2^{-3sb_{\Phi_{2}}/\rho}} \times \frac{1}{2^{jsa_{\Phi_{2}}/\rho}} \times \Phi_{2} \circ \Phi_{1}^{-1}\left(\frac{1}{|I_{j}|^{s}} \right).
\end{align*}
We deduce that $$ \Phi_{2}\left(  \Phi^{-1}_{1}\left(\dfrac{1}{y^{s}_{0}}\right) \frac{y_{0}^{2s/\rho}}{|\omega-\overline{z_{0}}|^{2s/\rho}}\right) \leq \frac{1}{2^{-3sb_{\Phi_{2}}/\rho}} \times \frac{1}{2^{jsa_{\Phi_{2}}/\rho}} \times \Phi_{2} \circ \Phi_{1}^{-1}\left(\frac{1}{|I_{j}|^{s}} \right), ~~ \forall~ \omega \in E_{j}. $$ 
Since the inequality (\ref{eq:ual65iaq5rsleson}) holds, it follows that
 \begin{align*}
\int_{E_{j}}\Phi_{2}\left(  \Phi^{-1}_{1}\left(\dfrac{1}{y^{s}_{0}}\right) \frac{y_{0}^{2s/\rho}}{|\omega-\overline{z_{0}}|^{2s/\rho}}\right) d\mu(\omega) &\leq \int_{E_{j}} \frac{1}{2^{-3sb_{\Phi_{2}}/\rho}} \times \frac{1}{2^{jsa_{\Phi_{2}}/\rho}} \times \Phi_{2} \circ \Phi_{1}^{-1}\left(\frac{1}{|I_{j}|^{s}} \right)d\mu(\omega)\\
&\leq \frac{1}{2^{-3sb_{\Phi_{2}}/\rho}} \times \frac{1}{2^{jsa_{\Phi_{2}}/\rho}} \times\Phi_{2} \circ \Phi_{1}^{-1}\left(\frac{1}{|I_{j}|^{s}} \right) \mu(Q_{I_{j}}) \\
&\leq \frac{1}{2^{-3sb_{\Phi_{2}}/\rho}} \times \frac{1}{2^{jsa_{\Phi_{2}}/\rho}} \times C_{1}.
\end{align*}
We deduce that  $$ \int_{E_{j}}\Phi_{2}\left(  \Phi^{-1}_{1}\left(\dfrac{1}{y^{s}_{0}}\right) \frac{y_{0}^{2s/\rho}}{|\omega-\overline{z_{0}}|^{2s/\rho}}\right) d\mu(\omega) \leq \frac{C_{1}}{2^{-3sb_{\Phi_{2}}/\rho}} \times \frac{1}{2^{jsa_{\Phi_{2}}/\rho}}, ~~ \forall~j\geq 0.       $$
By construction, the $E_{j}$ are pairwise disjoint and form a partition of $\mathbb{C}_{+}$. So
we have
\begin{align*}
\int_{\mathbb{C}_{+}}\Phi_{2}\left(  \Phi^{-1}_{1}\left(\dfrac{1}{y^{s}_{0}}\right) \frac{y_{0}^{2s/\rho}}{|\omega-\overline{z_{0}}|^{2s/\rho}}\right) d\mu(\omega) &= \sum_{j=0}^{\infty}	\int_{E_{j}}\Phi_{2}\left(  \Phi^{-1}_{1}\left(\dfrac{1}{y^{s}_{0}}\right) \frac{y_{0}^{2s/\rho}}{|\omega-\overline{z_{0}}|^{2s/\rho}}\right)  d\mu(\omega)\\
&\leq \frac{C_{1}}{2^{-3sb_{\Phi_{2}}/\rho}} \times \sum_{j=0}^{\infty}\frac{1}{2^{jsa_{\Phi_{2}}/\rho}} < \infty.
\end{align*}
\epf

\section{Proofs of main results.}
\proof[Proof of Theorem \ref{pro:main127}.]
The equivalence $(i)\Leftrightarrow (ii)$ is given by Theorem \ref{pro:main14paq1}. The implication $(iii)\Rightarrow (iv)$ is obvious. Let us prove that $(i)\Rightarrow (iii)$ and $(iv)\Rightarrow (i)$ which is enough to conclude.
\medskip

$(i)\Rightarrow (iii)$:
Let $0\not\equiv F \in H^{\Phi_{1}}(\mathbb{C_{+}})$. 
According to Theorem \ref{pro:mainfaaqaqaqq5}, there exists a unique function
   $f\in L^{\Phi}\left(\mathbb{R}\right)$ such that $\log|f| \in L^{1}\left(\frac{dt}{1+t^{2}}\right)$ and
\begin{equation}\label{eq:ibejtion}
\log|F(x+iy)| \leq   \frac{1}{\pi}\int_{\mathbb{R}}\frac{y}{(x-t)^{2}+y^{2}}\log|f(t)|dt, ~~\forall~x+iy \in \mathbb{C}_{+}\end{equation}
and 
$  \|F\|_{H^{\Phi}}^{lux}=\|f\|_{L^{\Phi}}^{lux}$.
Using Jensen's inequality in Relation (\ref{eq:ibejtion}), we deduce that
$$  |F(x+iy)| \lesssim \left(\mathcal{M}_{HL}(|f|^{a_{\Phi_{1}}/2}) (x)\right)^{2/a_{\Phi_{1}}},  ~~\forall~x+iy \in \mathbb{C}_{+}.    $$
Fix  $\lambda >0$ and put 
$$  E_{\lambda}:=\left\{ x\in \mathbb{R} : \left(\mathcal{M}_{HL}\left(\frac{|f|}{\|f\|_{L^{\Phi}}^{lux}}\right)^{a_{\Phi_{1}}/2} (x)\right)^{2/a_{\Phi_{1}}} >  \lambda \right\}.   $$
From the Relation(\ref{eq:of895ua1l}),  we deduce that
$$   |E_{\lambda}| \lesssim \sum_{\beta \in \{0; 1/3\} }  \left|\left\{ x\in \mathbb{R} : \left(\mathcal{M}_{HL}^{\mathcal{D}^{\beta}}\left(\frac{|f|}{\|f\|_{L^{\Phi}}^{lux}}\right)^{a_{\Phi_{1}}/2} (x)\right)^{2/a_{\Phi_{1}}} >  \frac{\lambda}{12} \right\} \right|.  $$
Put $$  \Phi_{a}(t)=\Phi_{1}\left(t^{2/a_{\Phi_{1}}}\right), ~~\forall~t\geq 0.      $$
From Proposition \ref{pro:main3aaqq7}, we deduce that  $\Phi_{a} \in \mathscr{U}\cap \nabla_{2}$. 
According to Proposition \ref{pro:main9aaq9}, it follows that $$ \left|\left\{ x\in \mathbb{R} : \left(\mathcal{M}_{HL}^{\mathcal{D}^{\beta}}\left(\frac{|f|}{\|f\|_{L^{\Phi}}^{lux}}\right)^{a_{\Phi_{1}}/2} (x)\right)^{2/a_{\Phi_{1}}} >  \frac{\lambda}{12} \right\} \right| \lesssim  \frac{1}{\Phi_{1}(\lambda)}, ~~\forall~ \beta \in \{0; 1/3\}.    $$
We deduce that $$  |E_{\lambda}| \lesssim  \frac{1}{\Phi_{1}(\lambda)}. $$
Put  
$$ \widetilde{\Omega}_{3}(t)=\frac{1}{\Phi_{2}\circ \Phi_{1}^{-1}\left(\frac{1}{t}\right)}, ~~\forall~t>0 \hspace*{0.5cm}\textrm{and} \hspace*{0.5cm}  \widetilde{\Omega}_{3}(0)=0. $$
From Lemma \ref{pro:main40}, we deduce that   $\widetilde{\Omega}_{3} \in \mathscr{U}$.
Since  $\mu$ is an $ \Phi_{2}\circ \Phi_{1}^{-1}-$Carleson measure and  $t\mapsto\frac{\widetilde{\Omega}_{3}(t)}{t}$ is non-decreasing on $\mathbb{R_{+}^{*}}$,  by Lemma \ref{pro:main132},  we have
\begin{align*}
\mu\left(\left\{ z\in \mathbb{C_{+}} : |F(z)| >  \lambda \|f\|_{L^{\Phi_{1}}}^{lux}\right\}\right) &\lesssim \mu\left(\left\{ z\in \mathbb{C_{+}} : |U_{f}(z)| >  \lambda \|f\|_{L^{\Phi_{1}}}^{lux}\right\}\right)\\
&\lesssim  \widetilde{\Omega}_{3}\left(\left|E_{\lambda}  \right|\right) \\
&\lesssim \Phi_{1}(\lambda) \widetilde{\Omega}_{3}\left(\frac{1}{\Phi_{1}(\lambda)} \right) \left|E_{\lambda}  \right|.
\end{align*}
As 
$$ \Phi_{1}(\lambda) \widetilde{\Omega}_{3}\left(\frac{1}{\Phi_{1}(\lambda)} \right) = \Phi_{1}(\lambda)\frac{1}{\Phi_{2}(\lambda)} = \frac{\Phi_{1}(\lambda)}{\lambda} \times \frac{\lambda}{\Phi_{2}(\lambda)} \approx \frac{\Phi_{1}'(\lambda)}{\Phi_{2}'(\lambda)}. $$
We deduce that
$$ \mu\left(\left\{ z\in \mathbb{C_{+}} : |F(z)| >  \lambda \|f\|_{L^{\Phi_{1}}}^{lux}\right\}\right) \lesssim  \frac{\Phi_{1}'(\lambda)}{\Phi_{2}'(\lambda)}\left|E_{\lambda}  \right|,~~\forall~\lambda >0.   $$
We have
\begin{align*}
\int_{\mathbb{C}_{+}}\Phi_{2}\left( \frac{|F(z)|}{\|F\|_{H^{\Phi_{1}}}^{lux}}\right)d\mu(z) &= \int_{0}^{\infty}\Phi_{2}^{\prime}(\lambda)\mu\left(\left\{ z\in \mathbb{C_{+}} : |F(z)| >  \lambda \|f\|_{L^{\Phi_{1}}}^{lux}\right\}\right)d\lambda     \\
&\lesssim  \int_{0}^{\infty}\Phi_{2}^{\prime}(\lambda)\left(\dfrac{\Phi_{1}^{\prime}(\lambda)}{\Phi_{2}^{\prime}(\lambda)} \times |E_{\lambda}| \right)d\lambda \\
&= \int_{0}^{\infty}\Phi_{1}^{\prime}(\lambda)\times |E_{\lambda}|d\lambda 
= \int_{\mathbb{R}}\Phi_{a}\left(  \mathcal{M}_{HL}^{\mathcal{D}^{\beta}}\left(\frac{|f|}{\|f\|_{L^{\Phi}}^{lux}}\right)^{a_{\Phi_{1}}/2} (x)  \right)dx \\
&\lesssim  \int_{\mathbb{R}}\Phi_{1}\left( \frac{|f(x)|}{\|f\|_{L^{\Phi_{1}}}^{lux}}\right)dx 
\lesssim 1.
\end{align*}
$(iv)\Rightarrow (i)$:  Let $I$ be an interval of nonzero length and $Q_{I}$ its Carleson square.\\
Fix $z_{0}=x_{0}+iy_{0} \in \mathbb{C}_{+}$ and we assume that  $x_{0}$ is the center of $I$ and $|I|=2y_{0}$. Put 
$$  F_{z_{0}}(\omega)=\Phi^{-1}_{1}\left(\frac{1}{\pi y_{0}}\right) \frac{y^{2/\rho}_{0}}{ (\omega-\overline{z_{0}})^{2/\rho}} ,~ \forall~ \omega\in \mathbb{C_{+}},           $$
where $\rho=1$ if $\Phi \in \mathscr{U}$ and $\rho=a_{\Phi}$ if  $\Phi \in \mathscr{L}$. By Proposition \ref{pro:main10aqa9}, we deduce that
$F_{z_{0}} \in  H^{\Phi_{1}}(\mathbb{C_{+}})$ and $\|F_{z_{0}}\|_{H^{\Phi_{1}}}^{lux}\leq 1$.\\
Let $\omega=u+iv \in Q_{I}$. We have 
$$ |\omega-\overline{z_{0}}|^{2}=|(u-x_{0})+i(v+y_{0})|^{2} \leq y_{0}^{2}+(2y_{0}+y_{0})^{2}=10y_{0}^{2}\Rightarrow \frac{1}{10} \leq \dfrac{y^{2}_{0}}{ |\omega-\overline{z_{0}}|^{2}}.   $$
Since the function $t\mapsto\frac{\Phi_{1}^{-1}(t)}{t^{1/\rho}}$ is non-increasing on $\mathbb{R}_{+}^{*}$, we have 
$$  \Phi^{-1}_{1}\left(\frac{1}{|I|}\right)<  \Phi^{-1}_{1}\left(\frac{1}{y_{0}}\right) \leq \pi^{1/\rho}\Phi^{-1}_{1}\left(\frac{1}{\pi y_{0}}\right).  $$
We deduce that 
$$  \Phi^{-1}_{1}\left(\frac{1}{|I|}\right)<  \left(\frac{\pi}{10}\right)^{1/\rho} \Phi^{-1}_{1}\left(\frac{1}{\pi y_{0}}\right)\frac{y_{0}^{2/\rho}}{ |\omega-\overline{z_{0}}|^{2/\rho}} \leq \left(\frac{\pi}{10}\right)^{1/\rho} \frac{|F_{z_{0}}(\omega)|}{\|F_{z_{0}}\|_{H^{\Phi_{1}}}^{lux}}.
  $$
Taking   
$$  \lambda:=\left(\frac{10}{\pi}\right)^{1/\rho}\Phi_{1}^{-1}\left(\frac{1}{|I|}\right),   $$
it follows that
 $$ |F_{z_{0}}(\omega)|>\lambda \|F_{z_{0}}\|_{H^{\Phi_{1}}}^{lux}, ~~ \forall~\omega \in Q_{I}.      $$
Therefore
 $$  Q_{I} \subset \left\{ z\in \mathbb{C_{+}} : |F_{z_{0}}(z)|>\lambda \|F_{z_{0}}\|_{H^{\Phi_{1}}}^{lux}  \right\}.    $$
Since inequality (\ref{eq:iberginjectio2}) is satisfied, we have 
$$ \mu(Q_{I}) \leq \mu\left( \left\{ z\in \mathbb{C_{+}} : |F_{z_{0}}(z)|>\lambda \|F_{z_{0}}\|_{H^{\Phi_{1}}}^{lux}\right\}\right) \leq \frac{C_{1}}{\Phi_{2}(\lambda)}.    $$
As 
$$ \Phi_{2} \circ \Phi_{1}^{-1}\left(\frac{1}{|I|} \right)= \Phi_{2}\left( \left(\frac{\pi}{10}\right)^{1/\rho} \lambda  \right)  \leq C_{2} \Phi_{2}(\lambda). $$
We deduce that 
$$ \mu(Q_{I}) \leq \frac{C_{3}}{\Phi_{2} \circ \Phi_{1}^{-1}\left(\frac{1}{|I|} \right)}.      $$
\epf

\proof[Proof of Corollary \ref{pro:main149}.]
The proof of Corollary \ref{pro:main149} follows from Theorem \ref{pro:main127} and Proposition \ref{pro:main14a9} for ($s=1$).
\epf

\proof[Proof of Theorem \ref{pro:main150}.]
The equivalence $(i)\Leftrightarrow (ii)$ is given by Theorem \ref{pro:main14paq1}. The implication $(iii)\Rightarrow (iv)$ is obvious. To conclude, it is enough to prove that  $(i)\Rightarrow (iii)$ and $(iv)\Rightarrow (i)$.
\medskip

$(i)\Rightarrow (iii)$:
Let  $0\not\equiv F \in A_{\alpha}^{\Phi_{1}}(\mathbb{C_{+}})$. By Proposition \ref{pro:main1aq18}, there exists $\beta \in \left\{0, 1/3\right \}$ such that
$$ |G(z)| \lesssim \left(\mathcal{M}_{V_{\alpha}}^{\mathcal{D}^{\beta}}\left(|G|^{a_{\Phi_{1}}/2}\right)(z)\right)^{2/a_{\Phi_{1}}}, ~~\forall~z \in \mathbb{C_{+}},
       $$
where $G:=\frac{|F(z)|}{\|F\|_{A_{\alpha}^{\Phi_{1}}}^{lux}}$. Put  
$$ \widetilde{\Omega}_{3}(t)=\frac{1}{\Phi_{2}\circ \Phi_{1}^{-1}\left(\frac{1}{t}\right)}, ~~\forall~t>0 \hspace*{0.5cm}\textrm{and} \hspace*{0.5cm}  \widetilde{\Omega}_{3}(0)=0. $$
From Lemma \ref{pro:main40}, we deduce that   $\widetilde{\Omega}_{3} \in \mathscr{U}$.
Since  $t\mapsto\frac{\widetilde{\Omega}_{3}(t)}{t}$ is non-decreasing on $\mathbb{R_{+}^{*}}$, according to Proposition \ref{pro:main9aaq9}, for  $\lambda >0$,  we have
$$    \left|E_{\lambda}  \right|_{\alpha} \leq \frac{1}{\Phi_{1}(\lambda)}  \Rightarrow  \widetilde{\Omega}_{3}\left(\left|E_{\lambda}  \right|_{\alpha}\right) \leq \Phi_{1}(\lambda) \widetilde{\Omega}_{3}\left(\frac{1}{\Phi_{1}(\lambda)} \right) \left|E_{\lambda}  \right|_{\alpha} \lesssim  \frac{\Phi_{1}'(\lambda)}{\Phi_{2}'(\lambda)}\left|E_{\lambda}  \right|_{\alpha}, $$
where 
$$  E_{\lambda}:=  \left\{z\in \mathbb{C}_{+} :  \left(\mathcal{M}_{V_{\alpha}}^{\mathcal{D}^{\beta}}\left(|G|^{a_{\Phi_{1}}/2}\right)(z)\right)^{2/a_{\Phi_{1}}} > \lambda \right\}.  $$
Since $\mu$ is an $(\alpha, \Phi_{2}\circ \Phi_{1}^{-1})-$Carleson measure, by Lemma \ref{pro:main132},  we deduce that $$ \mu(E_{\lambda}) \lesssim \widetilde{\Omega}_{3}\left(\left|E_{\lambda}  \right|_{\alpha}\right) \lesssim  \frac{\Phi_{1}'(\lambda)}{\Phi_{2}'(\lambda)}\left|E_{\lambda}  \right|_{\alpha},~~\forall~\lambda >0.   $$
Put $$  \Phi_{a}(t)=\Phi_{1}\left(t^{2/a_{\Phi_{1}}}\right), ~~\forall~t\geq 0.      $$
From Proposition \ref{pro:main3aaqq7}, we deduce that  $\Phi_{a} \in \mathscr{U}\cap \nabla_{2}$. We have
\begin{align*}
\int_{\mathbb{C}_{+}}\Phi_{2}\left( \frac{|F(z)|}{\|F\|_{A_{\alpha}^{\Phi_{1}}}^{lux}}\right)d\mu(z) 
 &\lesssim  \int_{\mathbb{C}_{+}}\Phi_{2}\left(\left(\mathcal{M}_{V_{\alpha}}^{\mathcal{D}^{\beta}}\left(|G|^{a_{\Phi_{1}}/2}\right)(z)\right)^{2/a_{\Phi_{1}}} \right)d\mu(z) \\
&= \int_{0}^{\infty}\Phi_{2}^{\prime}(\lambda)\mu(E_{\lambda}) d\lambda 
 \\
&\lesssim \int_{0}^{\infty}\Phi_{2}^{\prime}(\lambda)\left(\frac{\Phi_{1}'(\lambda)}{\Phi_{2}'(\lambda)}\left|E_{\lambda}  \right|_{\alpha}\right)d\lambda 
 \\
&= \int_{\mathbb{C_{+}}}\Phi_{a}\left(\mathcal{M}_{V_{\alpha}}^{\mathcal{D}^{\beta}}\left(|G|^{a_{\Phi_{1}}/2}\right)(z)\right)dV_{\alpha}(z)\\
&\lesssim \int_{\mathbb{C_{+}}}\Phi_{a}\left(|G|^{a_{\Phi_{1}}/2} \right)dV_{\alpha}(z) 
\lesssim 1.
\end{align*}
\medskip

$(iv)\Rightarrow (i)$:
Let $I$ be an interval of nonzero length and $Q_{I}$ its Carleson square.\\
Fix $z_{0}=x_{0}+iy_{0} \in \mathbb{C}_{+}$ and we assume that  $x_{0}$ is the center of $I$ and $|I|=2y_{0}$. Put 
$$  G_{z_{0}}(\omega)=\Phi_{1}^{-1}\left(\frac{1}{C_{\alpha}y_{0}^{2+\alpha}}\right) \frac{y_{0}^{(4+2\alpha)/\rho}}{ (\omega-\overline{z_{0}})^{(4+2\alpha)/\rho}},~ \forall~ \omega\in \mathbb{C_{+}},    $$
where $\rho=1$ if $\Phi \in \mathscr{U}$ and $\rho=a_{\Phi}$ if  $\Phi \in \mathscr{L}$, and $C_{\alpha}$ is the constant in the Relation (\ref{eq:fo4naqm1l}).
From the Proposition \ref{pro:main10aqa9}, we deduce that
$G_{z_{0}} \in  A_{\alpha}^{\Phi_{1}}(\mathbb{C_{+}})$ and $\|G_{z_{0}}\|_{A_{\alpha}^{\Phi_{1}}}^{lux}\leq 1$.\\
For $\omega=u+iv \in Q_{I}$, we have 
$$ |\omega-\overline{z_{0}}|^{2}=|(u-x_{0})+i(v+y_{0})|^{2} \leq y_{0}^{2}+(2y_{0}+y_{0})^{2}=10y_{0}^{2}\Rightarrow \frac{1}{10} \leq \dfrac{y^{2}_{0}}{ |\omega-\overline{z_{0}}|^{2}}.   $$
Since the function $t\mapsto\frac{\Phi_{1}^{-1}(t)}{t^{1/\rho}}$ is non-increasing on $\mathbb{R}_{+}^{*}$, we have 
$$  \Phi^{-1}_{1}\left(\frac{1}{|I|^{2+\alpha}}\right)<  \Phi^{-1}_{1}\left(\frac{1}{y_{0}^{2+\alpha}}\right) \leq (C_{\alpha})^{1/\rho}\Phi^{-1}_{1}\left(\frac{1}{C_{\alpha}y_{0}^{2+\alpha}}\right).  $$
We deduce that 
$$  \Phi^{-1}_{1}\left(\frac{1}{|I|^{2+\alpha}}\right)<  \left(\frac{C_{\alpha}}{10}\right)^{1/\rho} \Phi^{-1}_{1}\left(\frac{1}{C_{\alpha}y_{0}^{2+\alpha}}\right)\frac{y_{0}^{(4+2\alpha)/\rho}}{ |\omega-\overline{z_{0}}|^{(4+2\alpha)/\rho}} \leq \left(\frac{C_{\alpha}}{10}\right)^{1/\rho} \frac{|G_{z_{0}}(\omega)|}{\|G_{z_{0}}\|_{A_{\alpha}^{\Phi_{1}}}^{lux}}.
  $$
Taking   
$$  \lambda:=\left(\frac{10}{C_{\alpha}}\right)^{1/\rho}\Phi_{1}^{-1}\left(\frac{1}{|I|^{2+\alpha}}\right),   $$
it follows that
 $$ |G_{z_{0}}(\omega)|>\lambda \|G_{z_{0}}\|_{A_{\alpha}^{\Phi_{1}}}^{lux}, ~~ \forall~\omega \in Q_{I}.      $$
Therefore
 $$  Q_{I} \subset \left\{ z\in \mathbb{C_{+}} : |G_{z_{0}}(z)|>\lambda \|G_{z_{0}}\|_{A_{\alpha}^{\Phi_{1}}}^{lux}  \right\}.    $$
Since inequality (\ref{eq:ibeectio2}) is satisfied, we have 
$$ \mu(Q_{I}) \leq \mu\left( \left\{ z\in \mathbb{C_{+}} : |G_{z_{0}}(z)|>\lambda \|G_{z_{0}}\|_{A_{\alpha}^{\Phi_{1}}}^{lux}\right\}\right) \leq \frac{C_{1}}{\Phi_{2}(\lambda)}.    $$
As 
$$ \Phi_{2} \circ \Phi_{1}^{-1}\left(\frac{1}{|I|^{2+\alpha}} \right)= \Phi_{2}\left( \left(\frac{C_{\alpha}}{10}\right)^{1/\rho} \lambda  \right)  \leq C_{2} \Phi_{2}(\lambda). $$
We deduce that 
$$ \mu(Q_{I}) \leq \frac{C_{3}}{\Phi_{2} \circ \Phi_{1}^{-1}\left(\frac{1}{|I|^{2+\alpha}} \right)}.      $$
\epf

\proof[Proof of Corollary \ref{pro:main151}.]
The proof of Corollary \ref{pro:main151} follows from Theorem \ref{pro:main150} and Proposition \ref{pro:main14a9} for ($s=2+\alpha$).
\epf

The following result follows from the Lemma \ref{pro:main118} and the Proposition \ref{pro:main10aqa9}. Therefore, the proof will not be written.

\begin{lemme}\label{pro:main160}
Let $\alpha,\beta> -1$,   $\Phi_{1},\Phi_{2} \in \mathscr{L} \cup \mathscr{U}$. There are constants $C_{1}:=C_{\alpha,\Phi_{1},\Phi_{2}}>0$ and $C:=C_{\alpha,\beta,\Phi_{1}, \Phi_{2}}>0$ such that for all  $ F \in \mathcal{M}\left( H^{\Phi_{1}}(\mathbb{C_{+}}) ,  A^{\Phi_{2}}_{\alpha}(\mathbb{C_{+}})\right)$ and $ G \in \mathcal{M}\left( A_{\alpha}^{\Phi_{1}}(\mathbb{C_{+}}) ,  A^{\Phi_{2}}_{\beta}(\mathbb{C_{+}})\right)$, 
\begin{equation}\label{eq:ibergtion}
|F(x+iy)| \leq C_{1} \frac{\Phi_{2}^{-1}\left(\frac{1}{y^{2+\alpha}}\right)}{\Phi_{1}^{-1}\left(\frac{1}{y}\right)},~~ \forall~x+iy \in \mathbb{C_{+}} \end{equation}	
and 
\begin{equation}\label{eq:iberaqgt12}
|G(x+iy)| \leq C_{2} \dfrac{\Phi_{2}^{-1}\left(\frac{1}{y^{2+\beta}}\right)}{\Phi_{1}^{-1}\left(\frac{1}{y^{2+\alpha}}\right)}, ~~ \forall~x+iy \in \mathbb{C_{+}}.\end{equation}
\end{lemme}

\proof[Proof of Theorem \ref{pro:main16qs13}.]
The inclusion $\mathcal{M}( H^{\Phi_{1}}(\mathbb{C_{+}}) ,  A^{\Phi_{2}}_{\alpha}(\mathbb{C_{+}}))$ in $H_{\omega}^{\infty}(\mathbb{C_{+}})$ follows from Lemma \ref{pro:main160}.\\
Conversely, \\
Fix $ 0\not \equiv G \in H_{\omega}^{\infty}(\mathbb{C_{+}})$ and let $z=x+iy \in \mathbb{C_{+}}$.  
Since $\Phi_{2}\in \widetilde{\mathscr{L}}\cup\widetilde{\mathscr{U}}$, by Lemma \ref{pro:main80jdf}, we have
$$ \Phi_{2}(\omega(y)) =  \Phi_{2}\left( \frac{\Phi_{2}^{-1}\left(\frac{1}{y^{2+\alpha}}\right)}{\Phi_{1}^{-1}\left(\frac{1}{y}\right)}\right) \lesssim \frac{\Phi_{2}\left(\Phi_{2}^{-1}\left(\frac{1}{y^{2+\alpha}}\right)\right)}{\Phi_{2}\left(\Phi_{1}^{-1}\left(\frac{1}{y}\right)\right)} = \frac{1}{y^{2+\alpha}\Phi_{2} \circ \Phi_{1}^{-1}\left(\frac{1}{y}\right)}.      $$
We deduce that 
$$  \Phi_{2}\left(\frac{|G(x+iy)|}{\|G\|_{H_{\omega}^{\infty}}} \right)  \lesssim \Phi_{2}(\omega(y))\lesssim\frac{1}{y^{2+\alpha}\Phi_{2} \circ \Phi_{1}^{-1}\left(\frac{1}{y}\right)},~~ \forall~ x+iy \in \mathbb{C_{+}}.      $$
Put 
$$ d\mu(x+iy)=\dfrac{dx dy}{y^{2}\Phi_{2} \circ \Phi_{1}^{-1}(\frac{1}{y})},  ~~ \forall~ x+iy \in \mathbb{C_{+}}. $$
Since $\Phi_{2} \circ \Phi_{1}^{-1} \in \nabla_2$, from  Proposition \ref{pro:main49qa2}, we deduce that $\mu$ is a measure $\Phi_{2}\circ \Phi_{1}^{-1}-$Carleson.\\
Let $0\not\equiv F \in H^{\Phi_{1}}(\mathbb{C_{+}})$. By the Theorem \ref{pro:main127}, we have
\begin{align*}
\int_{\mathbb{C_{+}}}\Phi_{2}\left( \frac{|G(x+iy)F(x+iy)|}{\|G\|_{H_{\omega}^{\infty}}\|F\|_{H^{\Phi_{1}}} ^{lux}}\right) dV_{\alpha}(x+iy) 
 &\lesssim \int_{\mathbb{C_{+}}}\Phi_{2}\left(\frac{|G(x+iy)|}{\|G\|_{H_{\omega}^{\infty}}} \right) \Phi_{2}\left(\frac{|F(x+iy)|}{\|F\|_{H^{\Phi_{1}}} ^{lux}}\right) y^{\alpha}dxdy \\
&\lesssim \int_{\mathbb{C_{+}}} \Phi_{2}\left(\frac{|F(x+iy)|}{\|F\|_{H^{\Phi_{1}}} ^{lux}}\right) d\mu(x+iy)\\
&\lesssim 1. 
\end{align*}
We deduce that $G \in \mathcal{M}( H^{\Phi_{1}}(\mathbb{C_{+}}) ,  A^{\Phi_{2}}_{\alpha}(\mathbb{C_{+}}))$.
 \epf

\proof[Proof of Theorem \ref{pro:main1sr606}.]
 The inclusion $\mathcal{M}( A_{\alpha}^{\Phi_{1}}(\mathbb{C_{+}}) ,  A^{\Phi_{2}}_{\beta}(\mathbb{C_{+}}))$ in $H_{\omega}^{\infty}(\mathbb{C_{+}})$ follows from Lemma \ref{pro:main160}.\\
Conversely, \\
Fix $ 0\not \equiv G \in H_{\omega}^{\infty}(\mathbb{C_{+}})$ and let $z=x+iy \in \mathbb{C_{+}}$.  
Since $\Phi_{2}\in \widetilde{\mathscr{L}}\cup\widetilde{\mathscr{U}}$, by Lemma \ref{pro:main80jdf}, we have
$$ \Phi_{2}(\omega(y)) =  \Phi_{2}\left( \frac{\Phi_{2}^{-1}\left(\frac{1}{y^{2+\beta}}\right)}{\Phi_{1}^{-1}\left(\frac{1}{y^{2+\alpha}}\right)}\right) \lesssim \frac{\Phi_{2}\left(\Phi_{2}^{-1}\left(\frac{1}{y^{2+\beta}}\right)\right)}{\Phi_{2}\left(\Phi_{1}^{-1}\left(\frac{1}{y^{2+\alpha}}\right)\right)} = \frac{1}{y^{2+\beta}\Phi_{2} \circ \Phi_{1}^{-1}\left(\frac{1}{y^{2+\alpha}}\right)}.      $$
We deduce that
$$ \Phi_{2}\left( \frac{|G(x+iy)|}{\|G\|_{H_{\omega}^{\infty}}}\right) \lesssim \Phi_{2}(\omega(y))\lesssim \dfrac{1}{y^{2+\beta}\Phi_{2} \circ \Phi_{1}^{-1}\left(\frac{1}{y^{2+\alpha}}\right)}, ~~ \forall~ x+iy \in \mathbb{C_{+}}.         $$
Put
$$ d\mu(x+iy)=\dfrac{dx dy}{y^{2}\Phi_{2} \circ \Phi_{1}^{-1}(\frac{1}{y^{2+\alpha}})},  ~~ \forall~ x+iy \in \mathbb{C_{+}}. $$
By Proposition \ref{pro:main49qa2}, $\mu$ is a  $(\alpha,\Phi_{2}\circ \Phi_{1}^{-1})-$Carleson measure.
By the Theorem \ref{pro:main150}, we have
\begin{align*}
 \int_{\mathbb{C_{+}}}\Phi_{2}\left( \frac{|G(x+iy)F(x+iy)|}{\|G\|_{H_{\omega}^{\infty}}\|F\|_{A_{\alpha}^{\Phi_{1}}} ^{lux}}\right) dV_{\beta}(x+iy) 
&\lesssim \int_{\mathbb{C_{+}}}\Phi_{2}\left( \frac{|G(x+iy)|}{\|G\|_{H_{\omega}^{\infty}}}\right) \Phi_{2}\left(\frac{|F(x+iy)|}{\|F\|_{A_{\alpha}^{\Phi_{1}}} ^{lux}}\right) y^{\beta}dxdy\\
&\lesssim \int_{\mathbb{C_{+}}} \Phi_{2}\left(\frac{|F(x+iy)|}{\|F\|_{A_{\alpha}^{\Phi_{1}}} ^{lux}}\right) d\mu(x+iy)\\
&\lesssim 1.
\end{align*}
We deduce that  $G \in \mathcal{M}( A_{\alpha}^{\Phi_{1}}(\mathbb{C_{+}}) ,  A^{\Phi_{2}}_{\beta}(\mathbb{C_{+}})).$
\epf

\bibliographystyle{plain}
 
\end{document}